\pdfoutput=1	
\documentclass[
	final,	
	bigMargin = false,
	font = libertine,
	useBibTeX = true
	]{MyClass}
\usepackage{amscd}

\newtheoremstyle{theorem}
  {3pt}
  {3pt}
  {\itshape}
  {}
  {\scshape}
  {:}
  {.5em}
  {}

\theoremstyle{theorem}
\newtheorem{thm}{Theorem}[section]
\crefname{thm}{Theorem}{Theorems}

\newtheorem{lem}[thm]{Lemma}
\newtheorem{prop}[thm]{Proposition}
\newtheorem{comp}[thm]{Complement}
\theoremstyle{definition}

\newtheorem{defn}[thm]{Definition}

\theoremstyle{remark}
\newtheorem{rema}[thm]{Remark}

\def\id{\mathrm{Id}}
\def\Hol{\mathrm{Hol}}
\def\Hom{\mathrm{Hom}}
\def\Ho{\mathrm{H}}
\def\cchi{\chi}
\def\cchiz{\chi}
\def\QQQ{{Q_\PP}}
\numberwithin{equation}{section}

\def\lift{\lambda}
\def\hh{\mathfrak h}
\def\wilson{\rho}
\def\wilsonn{\rho_{\flat}}

\def\prin{\xi}

\def\mani{M}
\def\form{\mathcal A}

\def\GG{G}
\def\gg{\mathfrak g}

\def\ZZZ{H}
\def\PP{P}

\def\zzz{\mathfrak h}

\def\conn{A}

\def\adx{\mathrm{ad}_{\prin}}

\def\pois{\{\,\cdot\, , \,\cdot\, \}}
\def\Stab{Z}

\def\curv{\mathrm{curv}}
\def\ggcchi{\gg}
\def\vol{\mathrm{vol}}
\def\wuw{u}
\def\pio{\pi_1}
\def\MZ{\mathbb Z}
\def\vS{\vol_{\Sigma}}
\def\Hom{\mathrm{Hom}}
\def\GR{\Gamma_{\mathbb R}}

\def\GGGG{M}
\def\MS{M_{\Sigma}}
\def\wMS{\widetilde M_{\Sigma}}
\def\DDD{\Delta}
\newcommand{\T}{\mathrm{T}}

\newcommand{\ba}{\begin{eqnarray}}
\newcommand{\ea}{\end{eqnarray}}

\newcommand{\Conn}{\mathcal A_{\prin}}
\newcommand{\Gau}{\mathcal{G}_{\prin}}

\newcommand{\gau}{\liea{gau}_{\prin}}
\newcommand{\GauQ}{\mathcal{G}_{\prin,Q}}
\newcommand{\GauA}{\mathcal{G}_{\prin,A}}

\newcommand{\defeq}{:=}
\renewcommand{\iref}[1]{\ref{#1}}
\newcommand{\liea}[1]{\mathfrak{#1}}
\newcommand{\set}[1]{\{#1\}}
\newcommand{\equivClass}[1]{#1}

\newcommand{\LeftTrans}{\mathrm{L}}     
    
\newcommand{\AdAction}{\mathrm{Ad}}    
     
\renewcommand{\Conj}{\mathrm{Ad}}

\usepackage{xparse}
\NewDocumentCommand{\field}{m}{\mathbb{#1}}			
\NewDocumentCommand{\R}{}{\field{R}}

\newcommand\dif{\mathop{}\!\mathrm{d}}		

\newcommand{\TBundle}{\mathrm{T}}
\newcommand{\tangent}{\mathrm{T}}

\DeclarePairedDelimiter\norm{\lVert}{\rVert}

\newcommand\normDot{\norm{\,\cdot\,}}

\Title{\vspace{-5ex}\Large Yang-Mills moduli spaces over an orientable closed surface via Fr\'echet reduction}

\Author[
	affiliation = mpiAndItp,
	email = {tobias.diez@itp.uni-leipzig.de}
]{diez}{T. Diez}

\Author[
	affiliation = lille,
	email = {Johannes.Huebschmann@math.univ-lille1.fr}
]{huebschmann}{J. Huebschmann}

\Institution{mpiAndItp}{%
	Max Planck Institute for Mathematics in the Sciences
	Inselstra{\ss}e 22, 
	04103 Leipzig, Germany
	and
	Institut f\"ur Theoretische Physik,
	Universit\"at Leipzig,
	Postfach 100 920,
	04009 Leipzig, Germany%
}

\Institution{lille}{%
	USTL, UFR de Math\'ematiques,
	CNRS-UMR 8524, 
	Labex CEMPI (ANR--11--LABX--0007--01), 
	59655 Villeneuve d'Ascq Cedex, France%
}



\Abstract{Given a principal bundle on an orientable closed surface 
with compact connected structure group,
we endow the space of based gauge equivalence classes of
smooth connections relative to smooth based
gauge transformations with the structure of a Fr\'echet
manifold. Using Wilson loop holonomies and a certain characteristic class
determined by the topology of the bundle, we then impose  
suitable constraints
on that Fr\'echet manifold
that single out the
based gauge equivalence classes of central Yang-Mills connections
but do not directly involve the Yang-Mills equation.
We also explain how our theory yields 
the 
based and unbased gauge equivalence classes of all Yang-Mills connections
and deduce the stratified symplectic structure
on the space of unbased gauge equivalence classes
of central Yang-Mills connections.
The crucial new
technical tool is a slice analysis in the Fr\'echet setting.
}

\begin{document}
\maketitle
\enlargethispage{3.4\baselineskip}
\tableofcontents

\section{Introduction}

We rework and extend,
in the framework of
Fr\'echet manifolds,
 the approach of Atiyah and Bott~\cite{atibottw} 
to the moduli spaces of Yang-Mills connections on a closed Riemann surface.

Let $\GG$ be a compact connected Lie group 
and $\mani$  a Riemannian manifold, 
and
let $\prin \colon \PP \to \mani$
be a principal $G$-bundle.
Yang-Mills theory proceeds by constructing
moduli spaces of solutions of the relevant
partial differential equations
modulo gauge transformations.
In the case of a manifold $M$ of arbitrary finite dimension,
smooth solutions need not exist and, even if they exist,
constructing solutions is a rather delicate endeavor.
For example, Uhlenbeck established the local solvability in the 
Coulomb gauge \cite{MR648356}. 
The main difficulty in Yang-Mills theory
resides, perhaps, in the fact that
the space of connections is too big, with too much unwanted
gauge freedom or gauge ambiguity.

We will here concentrate on the case
where the base manifold $M$ is an orientable closed (compact) 
Riemann surface, and 
 we write this surface as
$\Sigma$
and denote by $\vol_{\Sigma}$
a (suitably normalized) volume form on $\Sigma$.
The situation now greatly simplifies.
Indeed, 
the central Yang-Mills connections $\conn$
on $\prin$ (the Yang-Mills connections having central curvature)
are characterized by the equation
\begin{equation}
\curv_{\conn}=-X_\prin\cdot\vol_{\Sigma},
\end{equation}
for a certain characteristic class $X_\prin$
in the center $\hh$ of the Lie algebra $\gg$ of $\GG$
determined by the topology of the bundle $\prin$.
Concerning the sign, see \cref{univex} below.
In a fundamental paper~\cite{atibottw},
Atiyah and Bott 
 showed that, in the case at hand, the Yang-Mills solution space
decomposes into a disjoint union of solution spaces 
of central Yang-Mills connections
for reductions of the structure group.
Furthermore, they described the structure of such a central Yang-Mills 
solution space
using infinite dimensional techniques including
symplectic reduction phrased in terms
of suitable Sobolev spaces.
They emphatically raised the issue of understanding the singularities
of this kind of space.
In the presence of singularities, the concept of a stratified symplectic space
due to \cite{MR1127479}
satisfactorily isolates
the singular structure of a space that arises
by symplectic reduction in finite dimensions.
This approach has so far not been available for 
a space which results from an infinite-dimensional construction,
such as
the solution space of central Yang-Mills connections.
In particular, 
a rigorous theory of Poisson brackets in infinite dimensions
is lacking
and functional analytic issues obscure the underlying geometry.
	
We circumvent these technical issues and reduce 
the analysis of the singularities of the moduli space of central 
Yang--Mills connections to a finite-dimensional problem.
	As a new approach to the issue of gauge freedom, we perform 
reduction by stages, first relative to based gauge transformations
and thereafter relative to the residual action of the structure group.
We deliberately write \lq\lq reduction by stages\rq\rq\ 
rather than  \lq\lq symplectic reduction by stages\rq\rq,
since
there is no obvious momentum mapping
for the group of based gauge transformations.
We therefore take the entire orbit space
relative to the action,
on the  space of smooth connections,
of  the Fr\'echet Lie group of based gauge transformations.
We show that this orbit space is a Fr\'echet manifold.
We then show that the Wilson loop mapping developed
by the second-named author in~\cite{MR1600534}
yields a smooth map
from the 
 Fr\'echet manifold
of based gauge equivalence classes of connections
to the product manifold $\GG^{2\ell}$ 
(technically the space of $\GG$-valued homomorphisms 
defined on the free group $F$ on a family of $2\ell$
chosen generators of the fundamental group of $\Sigma$)
where $\ell$ denotes the genus of $\Sigma$.
We then impose,
in terms of the holonomies around all contractible loops bounded by a disk,
 suitable constraints 
on that Fr\'echet manifold
which single out the based gauge equivalence classes
of smooth central Yang-Mills connections.
These constraints recover the central curvature condition
in \cite{atibottw}
 and thereby yield a 
replacement for the Yang-Mills equation.
We show that 
the Wilson loop mapping
(more precisely, the map \eqref{singleout2} below)
yields a homeomorphism 
between the space
of based gauge equivalence classes of smooth 
connections satisfying the appropriate constraint
(equivalently: space of based gauge equivalence
classes of smooth central Yang--Mills connections)
endowed with its Fr\'echet topology
and a suitable space of homomorphisms into the structure group $\GG$,
realized within the product manifold $\GG^{2\ell}$.
We also show that, relative to
the $\GG$-orbit stratifications,
 the homeomorphism under discussion
is an isomorphism of stratified spaces
which, on each stratum,
restricts to a diffeomorphism.
See \cref{singleout} 
below for details.

A crucial step in the proof of \cref{singleout} consists in establishing
the fact that the Wilson loop mapping 
from that space of based gauge equivalence classes 
to that space of homomorphisms into the structure group $G$
is a local homeomorphism.
The continuity of that Wilson loop mapping is a consequence
of the smoothness of the Wilson loop mapping defined on the ambient space.
A construction in ~\cite[Section~6]{atibottw}
yields the inverse of the map under discussion, see \cref{novel} below.
While, on the one hand, the continuity of that inverse
map is a consequence
of Uhlenbeck compactness \cite{MR648356},
that compactness theorem relies on Sobolev space techniques.
We avoid these techniques and develop a proof 
that merely involves 
Fr\'echet space techniques. The proof we give includes the statement of
the compactness theorem  for our Fr\'echet space situation
over a surface. See  \cref{ubc} and \cref{lem2} below.

Our basic tool is a slice analysis technique 
for groups of smooth gauge transformations in the Fr\'echet setting.
This technique was developed in~\cite{AbbatiCirelliEtAl1989}
and~\cite{MR849796}
and has thereafter been extended by the first-named author
in~\cite{diez2013}.
It is
phrased within
the differential calculus due to
Michal and Bastiani, see~\cite{Neeb2006} for  further remarks
concerning the differential calculus on locally convex spaces,
and
crucially involves the version of the Nash--Moser inverse function theorem 
given in~\cite{MR656198} in the tame Fr\'echet setting.
Below we reproduce that result as \cref{prop::locallyConvexSpace:NashMoserInverseTheorem}.
This differential calculus, in turn,
has since been applied to
global analytic problems 
in~\cite{MR656198},~\cite{Milnor1984},~\cite{Neeb2006} and elsewhere.
The slice analysis technique
enables us to reconstruct, in our framework,
the stratified symplectic structure
on such a  Yang-Mills solution space
developed previously~\cite{MR1370113},~\cite{MR1277051}.
The upshot of the present paper is that
the theory developed in~\cite{MR1363857},
\cite{MR1369463}, and \cite{MR1600534} -- see~\cite{MR1938554} for 
a leisurely introduction -- can
be built in terms of the appropriate spaces
of \emph{smooth} connections.
Pushing a bit further, we endow the
Fr\'echet manifold of based gauge equivalence classes of connections
with a quasi-hamiltonian $G$-space structure
relative to the structure group $G$, and
the Yang-Mills moduli space we are interested in
then arises by reduction relative to $G$.
See \cref{further} below for details.
This observation, together
with the idea of imposing suitable constraints
on the space of based gauge equivalence classes of smooth connections,
justifies, perhaps, the
term {\em Fr\'echet reduction\/} in the title.
The analysis of the topology of the Yang-Mills solution spaces
in~\cite[Section 6]{atibottw}
enables us to reconstruct,
entirely in terms of Fr\'echet topologies on spaces
of smooth connections,
 the moduli space
of all Yang-Mills connections
on the fixed bundle $\prin$
and, furthermore,
that of all Yang-Mills connections
relative to $G$
over the surface $\Sigma$.
See \cref{all} for details.
Thus, we can understand
the Yang-Mills solution spaces
over a closed surface
in terms of spaces
of ordinary smooth connections.
This raises the issue whether 
we can understand the Yang-Mills solution spaces
over a general compact Riemannian manifold
in terms of spaces
of smooth connections.

For a general gauge theory situation,
an alternate model of the space of based holonomies
was developed by the second-named author in~\cite{MR1670408}.
That model 
yields a rigorous approach to lattice gauge theory.
For the special case explored in the present paper,
the alternate model comes down to the extended moduli space construction
 developed in~\cite{MR1277051},~\cite{MR1370113},~\cite{MR1470732},
see \cref{strati} below.
The construction
in the present paper is somewhat  a special case of
that in~\cite{MR1670408}, but now in the framework of
Fr\'echet manifolds.

In \cref{univex}, following \cite{atibottw},
 we spell out the universal example
in terms of the (necessarily unique) Schur cover of  the fundamental group
of $\Sigma$,
tailored to our purposes and,
in \cref{top1}, we recall 
the requisite 
topological classification of principal bundles over $\Sigma$.
In \cref{frechet} we introduce the Fr\'echet manifold structure
on the space of based equivalence classes of connections.
In \cref{wilson} we 
explore the smoothness of the Wilson loop mapping
and in \cref{preparing} we establish a technical result.
In \cref{constraints} we then proceed to
describe the constraints and to
 spell out and prove the 
main result, \cref{singleout}, 
and we complete
the proof 
in \cref{dep}.
In \cref{all}
we explain how we can recover,
entirely in terms of Fr\'echet topologies on spaces
of smooth connections,
 the moduli space
of all Yang-Mills connections relative to the structure group $G$.
In \cref{strati} we  briefly recall the resulting
singular symplectic geometry, phrased in the language of 
stratified symplectic spaces, 
and the final section, \cref{frechetslices}, is devoted to the requisite
Fr\'echet space technology.

\section{The universal example}
\label{univex}

We maintain the notation $G$ for a compact connected Lie group.
Our base manifold is an orientable 
(real) closed (compact) connected
surface $\Sigma$ of genus $ \geq 1$.
Our notation for the differential forms on a smooth manifold $M$ is
$\form(M, \,\cdot\,)$.

Ordinary Yang-Mills theory necessitates a choice
of Riemannian metric on the structure group
and on the base manifold.
Our approach does not involve
such metrics except when we link it to
ordinary Yang-Mills theory, and a choice of complex structure 
of $\Sigma$
will play no role either, except possibly for the
orientation it determines.
As a side remark we note that the Yang-Mills equation $\dif_A\ast\curv_A=0$
makes sense for 
a principal bundle having as structure group a general Lie 
group when we interpret 
the operator $\ast$
as having its values in the
forms with values in the dual to the adjoint bundle.

Let $\ell\, (\geq 1)$ denote the genus of $\Sigma$,
let $Q$ be a point of $\Sigma$, 
fixed throughout and taken henceforth as base point,
and let
$\pio=\pi_1(\Sigma,Q)$ denote 
the fundamental group of $\Sigma$
at the point $Q$.
Consider the standard presentation
\begin{equation}
\big\langle x_1,y_1,\dots, x_\ell,y_\ell; r\big\rangle,
\quad
r = \left[x_1,y_1\right] \cdot
\ldots
\cdot
\left[x_\ell,y_\ell\right]
\label{3.1}
\end{equation}
of $\pio$.
Let $u_1,v_1,\dots,u_\ell,v_\ell$
be a {\em canonical system\/} (in the sense of \cite{MR606743}) 
of (closed) curves
in $\Sigma$
that have $Q$ as starting point; that is to say:
Cutting $\Sigma$ successively along 
these curves yields a planar disk $e^2$;
see, e.g., \cite[\S 3.3.2]{MR606743} for details.
We suppose things arranged in such a way that these curves
represent the respective generators
$x_1,y_1,\dots,x_\ell,y_\ell$
of the fundamental group $\pio=\pi_1(\Sigma,Q)$ of $\Sigma$
and that the boundary path of the 
disk $e^2$ yields the relator $r$.
Thus the standard cell decomposition of
$\Sigma$ with a single 2-cell $e^2$
corresponding to $r$ results. 

We denote by $F$ the free group on
$x_1,y_1,\dots, x_\ell,y_\ell$ and by $N$ the normal closure
of $r$ in $F$. Consider the quotient group $\Gamma = F\slash [F,N]$.
The image $[r] \in \Gamma$ of $r$ 
generates the central subgroup 
$\mathbb Z\langle [r]\rangle =N\slash [F,N]$ of $\Gamma$,
 and
the resulting extension
\begin{equation}
\begin{CD}
0
@>>>
\mathbb Z\langle [r]\rangle
@>>>
\Gamma
@>>>
\pio
@>>>
1
\end{CD}
\label{3.2}
\end{equation}
is central.
Since the transgression homomorphism
$\Ho_2(\pio) \to \MZ\langle [r]\rangle$
is an isomorphism, the extension \eqref{3.2}
is a maximal stem extension (Schur cover) and since,
furthermore, the abelianization of $\pio$ 
is a free abelian group,
that maximal stem extension is unique
up to within isomorphism \cite[\S 9.9 Theorem 5 p.~214]{MR0279200}.
Atiyah and Bott use the terminology 
\lq\lq universal central extension\rq\rq\ 
to refer to this situation
\cite[\S 6]{atibottw}.

Relative to the embedding 
$\mathbb Z\langle [r]\rangle \to \mathbb R$ of 
abelian groups
given by the assignment to $[r]$ of $2 \pi\in \mathbb R$, 
similarly as in~\cite[\S 6 p.~560]{atibottw},
let $\GR$ denote
the 1-dimensional Lie group characterized by the requirement that
\begin{equation}
\begin{CD}
0
@>>>
\mathbb Z\langle [r]\rangle
@>>>
\Gamma
@>>>
\pio
@>>>
1
\\
@.
@VVV
@VVV
@|
@.
\\
0
@>>>
\mathbb R
@>>>
\GR
@>>>
\pio
@>>>
1
\end{CD}
\label{uce2}
\end{equation}
be a commutative diagram of central extensions.
Consider
a principal 
$\mathrm U(1)$-bundle $\prin_{\MS} \colon {\MS} \to \Sigma$
on $\Sigma$ having Chern class $1$,
cf., e.g., \cite[I.4 p. 58 ff]{hirzeboo} for the precise meaning
of having Chern class $1$.
Similarly as in \cite[\S 6]{atibottw}, we realize \eqref{3.2} 
in terms of $\prin_{\MS}$ as follows 
(in \cite[\S 6]{atibottw}, the compact 3-manifold $\MS$ is written as $Q$):
Choose
a smooth connection form 
\begin{equation}
\omega_{\MS}\colon \mathrm T{\MS} \longrightarrow i \mathbb R
=\mathrm{Lie}(\mathrm U(1)) \subseteq
\mathrm{Lie}(\mathrm{GL}(1,\mathbb C)) =\mathbb C 
\end{equation} 
on $\prin_{\MS}$
having non-degenerate curvature form
and write
the curvature form as
$-2 \pi i \vS \in \form^2(\Sigma, i\mathbb R)$.
This defines the volume form
$\vS$ on $\Sigma$ and fixes an orientation of $\Sigma$
and, when we orient $e^2$ consistently with $\Sigma$,
fixes an orientation
of the closed path 
$
{\left[u_1,v_1\right] \cdot
\ldots
\cdot
\left[u_\ell,v_\ell\right]}$,
the \lq\lq boundary path\rq\rq\ of
$e^2$, as well.
Unlike the approach in~\cite{atibottw}, in our setting,
there is no need to 
choose a Riemannian metric on $\Sigma$ and to
arrange for
$\vS$
to be the volume form associated to that Riemannian metric.
We can, of course, choose a (positive) complex structure on $\Sigma$
that is 
compatible with the volume form $\vS$ (as a symplectic structure)
and work with the resulting K\"ahler metric on $\Sigma$.
Then the $2$-form $\vS$ 
is the corresponding
normalized Riemannian volume form on $\Sigma$,
and the orientation of $\Sigma$ coincides with the orientation 
arising from the complex structure.
This reconciles our approach with that  in~\cite{atibottw}
and shows that the sign in the above expression for the curvature
form is consistent with standard Chern-Weil theory, cf. \cite{MR2244174}.

The group $\Gamma$ is isomorphic to the fundamental group of ${\MS}$. 
To make the identification
of $\Gamma$ with the fundamental group of ${\MS}$ explicit,
we choose a pre-image $Q_{\MS}\in {\MS}$ of the chosen base point $Q$
of $\Sigma$ and, thereafter,
closed lifts 
 ${u}_{1,{\MS}}$, ${v}_{1,{\MS}}$, \ldots, ${u}_{\ell,{\MS}}$, ${v}_{\ell,{\MS}}$, to ${\MS}$, 
 of the canonical curves having $Q_{\MS}\in {\MS}$ as starting point.
Though this is not strictly necessary, we note that,
since, up to gauge transformations, 
the curvature determines the connection only up to a member
of $\Ho^1(\Sigma,\mathrm U(1))$,
we can rechoose the smooth principal $\mathrm U(1)$-connection
form
$\omega_{\MS}$ in such a way
that its holonomies with respect to 
$Q_{\MS}$ and the  canonical curves in $\Sigma$ 
are trivial.
See also \cref{rem1} below.
Such a connection is unique up to gauge transformations. 
The respective horizontal lifts
 ${u}_{1,{\MS}}$, ${v}_{1,{\MS}}$, \ldots, ${u}_{\ell,{\MS}}$, ${v}_{\ell,{\MS}}$ 
in ${\MS}$
having $Q_{\MS}$ as starting point are then closed, and we can 
take their based homotopy classes as generators of $\pi_1({\MS},Q_{\MS})\cong \Gamma$.

We take the universal covering projection $\mathbb R \to \mathrm U(1)$
to be given by  the association $t \mapsto \mathrm e^{it}$ 
($t \in \mathbb R$).
The $\mathrm U(1)$-bundle projection $\prin_{\MS}\colon \MS \to \Sigma$
lifts to a principal $\mathbb R$-bundle projection
$\prin_{\wMS}\colon \wMS \to \widetilde \Sigma$
defined on the universal covering manifold
$\wMS$ of $\MS$ having as base the universal covering
manifold $\widetilde \Sigma$ of $\Sigma$,
and the $\mathbb R$-action and $\Gamma$-action on
$\wMS$ combine to a $\GR$-action on $\wMS$
turning the 
resulting projection $\prin_{{\wMS}}\colon {\wMS} \to \Sigma$
into
a principal bundle having $\GR$ as its structure group.
The connection form $\omega_{\MS}$ on $\prin_{\MS}$ determines a
connection form
$\omega_{{\wMS}} \colon \mathrm T{\wMS} \to \mathbb R$
on $\prin_{{\wMS}}\colon {\wMS} \to \Sigma$
and hence a connection
having curvature form
$-2 \pi  \vS \in \form^2(\Sigma, \mathbb R)$, and
this connection is necessarily central.
Here and henceforth we refer to a connection as being 
{\em central\/} when the values of its curvature form
lie in the center of the structure group.
The universal covering projection $\mathbb R \to \mathrm U(1)$
induces the map $\mathrm{Lie}(\GR)= \mathbb R \to i \mathbb R = \mathrm{Lie}(\mathrm U(1))$ given by $t \mapsto it$ ($t \in \mathbb R$) and,
by construction, the connection and curvature forms on
$\prin_{{\wMS}}$ and
$\prin_{{\MS}}$ correspond via that induced map.

\section{Topology of principal bundles over a closed oriented surface}
\label{top1}

To unveil, in the Fr\'echet setting, the structure of 
 the
moduli space of Yang-Mills connections on a fixed principal $G$-bundle
over $\Sigma$
we  
exploit a certain topological characteristic class 
that lies in the center  of the Lie algebra of $G$,  
cf.~\cite[\S 6 p.~560]{atibottw}.
For intelligibility and for later reference, 
we briefly recall how this class arises
and how it  relates to the topological classification
of principal $G$-bundles on $\Sigma$.

By structure theory, the compact connected Lie group $\GG$ is generated by
its maximal semi-simple subgroup $S=[\GG,\GG]$ and the connected component
$H$ of the identity of the center of $\GG$ and, accordingly, we write
$\GG$ as
$\GG=H \times_D S$, for the finite discrete central subgroup $D= H \cap S$
of $\GG$. 
The injection $H \to \GG$
induces an isomorphism $\overline H = H/D \to \GG/S$
onto the abelianized group $\GG_{\mathrm{Ab}}=\GG/S$, and
the group $\overline H$ is a torus group.

Let $\hh$ denote the Lie algebra 
of $\overline H$ (and of $H$).
The
assignment to $\varphi \in \Hom(\mathrm U(1),\overline H)$ of
the induced principal $\overline H$-bundle 
$\prin_{\MS} \times_{\varphi}\overline H\colon {\MS} \times_{\varphi}\overline H \to \Sigma$ topologically
classifies principal $\overline H$-bundles on $\Sigma$.
When we assign to a member $\varphi$ of $\Hom(\mathrm U(1),\overline H)$,
viewed as a \( 1 \)-parameter subgroup of \( \overline{H} \),
its tangent vector $X_\varphi \in \mathfrak h$ at the origin
(so that $\varphi(\mathrm e^{i t})=\mathrm{exp}(tX_\varphi)$),
we obtain an isomorphism
\begin{equation}
\Hom(\mathrm U(1),\overline H) \longrightarrow 
\mathrm{ker}(\mathrm{exp}\colon \mathfrak h \to \overline H)= \pi_1(\overline H)
\end{equation}
of discrete abelian groups.
Observing that
Poincar\'e duality with respect to the orientation of $\Sigma$
chosen in \cref{univex} above
identifies
$\mathrm H^2(\pio, \pi_1(\overline H))$ canonically with
$\pi_1(\overline H)$
reconciles this description with the ordinary topological classification
of principal $\overline H$-bundles on $\Sigma$.
In this language, our reference bundle
$\xi_{\MS}\colon \MS \to \Sigma$
introduced in \cref{univex} above has characteristic class
$X_{\xi_{\MS}}=2 \pi i \in i \R = \mathrm{Lie}(\mathrm U(1))$.

Consider a  principal $\GG$-bundle 
$\prin\colon \PP \to \Sigma$ on $\Sigma$.
The bundle projection factors through 
the induced map from the orbit manifold
$\PP/S$ to $\Sigma$, necessarily a smooth
principal $\overline H$-bundle,
and we denote this bundle by
$\prin_{\overline H}\colon \PP/S \to \Sigma$
and
by $X_\prin \in \hh$
the single topological characteristic class of
$\prin_{\overline H}$.
This yields the requisite topological characteristic class for
the bundle $\prin$, cf.~\cite[\S 6 p.~560]{atibottw}.

Consider the relator map
\begin{equation}
r\colon
\Hom(F,\GG)
\longrightarrow
\GG,
\ 
r(\chi)=\left[\chi(x_1),\chi(y_1)\right] \cdot
\dots
\cdot
\left[\chi(x_\ell),\chi(y_\ell)\right] \in \GG.
\label{relator}
\end{equation}
The injection  \( \Hom(\Gamma,\GG) \subseteq \Hom(F,\GG) \)
induced by the surjection $F \to \Gamma$ identifies 
a certain subspace
 of \( \Hom(\Gamma,\GG) \) with 
the subspace $r^{-1}(\mathrm{exp}(X_{\prin}))$
of  $\Hom(F,\GG)$, 
and we denote that subspace of \( \Hom(\Gamma,\GG) \)
by \( \Hom_{X_\prin}(\Gamma,\GG) \).
Viewed as  
a subspace of $\Hom(F,\GG)$, the
space $\Hom_{X_\prin}(\Gamma,\GG)$
is necessarily compact.

Let $\Hom_{X_\prin}(\GR,\GG)$ denote the space of homomorphisms
$\cchi$ from $\GR$ to $\GG$
that have the property that
$\cchi(t[r])=\exp(tX_{\prin})$ ($t \in \mathbb R$).
Given $\cchi\in \Hom_{X_\prin}(\Gamma,\GG)$, 
with a slight abuse of the notation
$\cchi$,
setting
$\cchi(t[r])=\exp(tX_{\prin})$ ($t \in \mathbb R$),
we obtain an extension 
of $\cchi$ to a homomorphism
$\cchi\colon \GR \longrightarrow \GG$,
uniquely determined by $\cchi$ and $X_{\prin}$, 
that is, 
the  restriction 
$\Hom_{X_\prin}(\GR,\GG) \to \Hom_{X_\prin}(\Gamma,\GG)$
is a bijection.
Hence, given $\cchi\in \Hom_{X_\prin}(\Gamma,\GG)$, 
the associated principal $\GG$-bundle 
$\prin_{\cchi}\colon {\wMS} \times_{\cchi} \GG \to \Sigma$
on $\Sigma$ is defined.

\begin{prop}
\label{prop1}
The assignment to 
$\cchi \in \Hom_{X_\prin}(\Gamma,\GG)$ 
of $\prin_{\cchi}$ yields a bijection between
the connected components of 
$\Hom_{X_\prin}(\Gamma,\GG)$ 
and the topological types
of  principal $\GG$-bundles on $\Sigma$ having
$X_{\prin} \in \hh$
as its corresponding characteristic class,
the number of components being given by the order
$|\pi_1(S)|$ of $\pi_1(S)$.
\end{prop}

The reasoning in~\cite[Section 6]{atibottw}
yields a proof of this proposition.
Suffice it to note the following:
Let $\overline S = \GG/H \cong S/D$.
The bundle $\prin$ is topologically classified by its characteristic class
in $\Ho^2(\Sigma, \pi_1(\GG))$ and, likewise, the principal
$\overline H$-bundle $\prin_{\overline H}$ 
is topologically classified by its characteristic class
in $\Ho^2(\Sigma, \pi_1(\overline H))$.
Under the induced homomorphism 
$\Ho^2(\Sigma, \pi_1(\GG))\to \Ho^2(\Sigma, \pi_1(\overline H))$,
the characteristic class of $\prin$ 
goes to the characteristic class of $\prin_{\overline H}$.
Poincar\'e duality relative to the fundamental class
$[\Sigma]$ identifies
that homomorphism with the induced homomorphism
$\pi_1(\GG)\to \pi_1(\overline H)$,
and the
induced homomorphism $\pi_1(S)\to \pi_1(\GG)$ identifies  
that kernel  with $\pi_1(S)$.
We leave the details to the reader.

We denote by $\Hom_{X_\prin}(\Gamma,\GG)_{\prin} \subseteq \Hom_{X_\prin}(\Gamma,\GG)$
the connected component that corresponds to the principal $\GG$-bundle $\prin$.

\section[Manifold structure on the space of based 
gauge equivalence classes of connections]{Fr\'echet manifold structure on the space of based 
gauge equivalence classes of connections} 
\label{frechet}

Consider a principal $\GG$-bundle $\prin \colon \PP \to \GGGG$
on a smooth compact connected manifold $\GGGG$.
Let $\gg$ denote the Lie algebra of $\GG$.
We use the standard formalism with structure group $\GG$
acting on $\PP$ from the right etc. 
We denote by
\( {\adx\colon \PP \times _\GG \gg \to \GGGG} \)
the (infinitesimal) adjoint bundle associated to \( \prin \) and  write 
the graded object of differential
forms on \( \GGGG \) with values in \( \adx \)
as
\(
\form^*(\GGGG, \adx) \).   

Consider the space $\Conn$ of {\em smooth $G$-connections\/}
on $\prin\colon \PP \to \GGGG$.
The Atiyah sequence~\cite{MR0086359}
associated with $\prin$, spelled out here for
the total spaces, has the form
\begin{equation}
0
\longrightarrow
\PP\times_G\gg
\longrightarrow
(\T \PP)/G
\longrightarrow
\T \GGGG
\longrightarrow
0,
\label{ati1}
\end{equation}
and we realize the space $\Conn$ as that of vector bundle sections
$
\T \GGGG
\to
(\T \PP)/G
$
for~\eqref{ati1}.
Thus the points of $\Conn$
are
in bijective correspondence with the sections of an affine bundle over 
\( \GGGG \)
and in this way carry a natural tame Fr\'echet manifold structure 
modeled on the
vector space 
\( \mathcal A^1(\GGGG, \adx) \)~\cite{AbbatiCirelliEtAl1986}.

Let $\Gau$ denote the group of {\em smooth gauge transformations\/} of
$\prin$. We realize this group as the group
\(
\Gamma^\infty(\Conj_\prin) \)
of smooth sections of the associated 
adjoint
bundle
$\Conj_\prin \colon \PP \times_G G \to \GGGG$,
endowed with the 
pointwise group structure.
In this manner, $\Gau$ acquires
the structure of
a tame Fr\'echet Lie group~\cite{AbbatiCirelliEtAl1986}. 

As usual, we identify the Lie algebra of $\Gau$  with the 
Lie algebra 
\( \gau =\Gamma^\infty(\adx) \) of smooth sections of the adjoint 
bundle $\adx$ on $\GGGG$, endowed with the 
pointwise Lie bracket.
This Lie algebra acquires
the structure of
a tame Fr\'echet Lie algebra
and, relative to the Fr\'echet structures, 
the exponential map \( \exp\colon \gau \to \Gau \) is a
local diffeomorphism at \( 0 \), see~\cite[Section~IV]{MR814813}.
In these topologies, the standard 
left
$\Gau$-action on $\Conn$ is tame smooth, proper and admits slices
at every point~\cite{AbbatiCirelliEtAl1989, diez2013}.
In particular the orbit space \( \Gau \backslash\Conn  \) is 
stratified by tame Fr\'echet manifolds.

Choose a base point  $\QQQ$ for the total space $\PP$ of
$\prin \colon \PP \to \GGGG$ and let
${Q=\prin(\QQQ) \in \GGGG}$. 
The \emph{evaluation map}
$\mathrm{ev}_{\QQQ}\colon \Gau \to G$ characterized by the identity
\[
\phi(Q) = [\QQQ, \mathrm{ev}_{\QQQ}(\phi)] \in \PP\times_G G,
\]
as  $\phi$ ranges over  $\Gau$ (the group of sections of $\Conj_\prin$),
is a morphism of Lie groups.
The  group of
\emph{smooth based gauge transformations} 
associated to the data
is the kernel of $\mathrm{ev}_{\QQQ}$. This group is 
a normal, locally exponential Lie subgroup of 
$\Gau$~\cite[Prop. IV.3.4]{Neeb2006}. 
While 
$\mathrm{ev}_{\QQQ}$ depends on the choice of
$\QQQ$, the kernel of $\mathrm{ev}_{\QQQ}$
 is independent of the particular choice of $\QQQ$
in $\prin^{-1}(Q)$ and depends only on $Q$;
we therefore denote the group of based gauge transformations by $\GauQ$.

Relative to $\Gau$, the group $\GauQ$ has finite
codimension equal to the dimension of $G$. \Cref{prop::lieGroup:finiteDimSubgroupIsStrongSplittingLieSubgroup} 
below entails that the obvious surjection
\( \Gau \to \Gau \slash \GauQ \) is a smooth right principal \( \GauQ
\)-bundle;
we refer to this kind of situation by the phrase
\lq\lq $\GauQ$
is a principal Lie
subgroup of $\Gau$\rq\rq.

\Cref{slice::actionOfSubgroup} below will say that
slices for the $\Gau$-action on $\Conn$ yield slices for the
 $\GauQ$-action on $\Conn$  and thus implies the
following:

\begin{prop}
\label{l1} 
Relative to the Fr\'echet manifold structures, the $\GauQ$-action on $\Conn$
is smooth and free; furthermore, it admits slices at every point. Hence the orbit space
$
\mathcal B_{\prin,Q}=\GauQ \backslash \Conn 
$ 
acquires a smooth Fr\'echet manifold structure, and 
the canonical projection
$\Conn \to \mathcal
B_{\prin,Q}$ is a smooth principal $\GauQ$-bundle.
\end{prop}

\section{Wilson loops}
\label{wilson}

Consider a principal $\GG$-bundle $\prin \colon \PP \to \GGGG$
on a smooth connected manifold $\GGGG$, not necessarily compact.
Choose a base point 
$\QQQ$ of $\PP$ and
consider a smooth closed path 
$\wuw\colon [0,1]\to \GGGG$
in $\GGGG$ starting at the point $\prin(\QQQ)$ of $\GGGG$. 
With respect to $\QQQ$
and a smooth connection $\conn$ on $\prin\colon \PP \to \GGGG$,
we  denote
by
$\mathrm{Hol}_{\wuw,\QQQ}(\conn)\in \GG$
the holonomy of $\conn$ along $\wuw$.
For later reference, we recall a description, 
cf.~\cite[Appendix to Lecture 3]{MR1338391}, \cite[Proof of Proposition II.3.1 p. 69]{MR0152974}.

Let $\wuw_\PP \colon [0,1] \to \PP$ be a lift of $\wuw$ having $\QQQ$
as its starting point, and let
$\omega_\conn$ denote the connection form associated to $\conn$. 
The solution
$\wuw_\GG\colon [0,1] \to \GG$  of the differential equation
\begin{equation}
\omega_\conn({(\wuw_\PP \wuw_\GG)\,\dot {} }\, )=0,\ \wuw_\GG(0)=e,
\label{diff3}
\end{equation}
then yield the desired horizontal 
lift $\wuw_\conn\colon [0,1] \to \PP$ 
of $\wuw$
as
$\wuw_\conn = \wuw_\PP \wuw_\GG$.

Now $(\wuw_\PP \wuw_\GG)\,\dot{}
= \dot \wuw_\PP \wuw_\GG +\wuw_\PP \dot \wuw_\GG$, 
the sum being evaluated, for given $t\in [0,1]$, in the vector space
$\mathrm T_{\wuw_\PP(t) \wuw_\GG(t)}\PP$,
and
\begin{equation}
\omega_\conn(\dot \wuw_\PP \wuw_\GG +\wuw_\PP \dot\wuw_\GG)
=\omega_\conn(\dot \wuw_\PP \wuw_\GG)+\omega_\conn(\wuw_\PP \dot\wuw_\GG)=
\mathrm{Ad}_{\wuw_\GG}^{-1}\omega_\conn(\dot \wuw_\PP)+\wuw_\GG^{-1}\dot\wuw_\GG,
\end{equation}
whence \eqref{diff3} is equivalent to
\begin{equation}
\dot \wuw_\GG  \wuw^{-1}_\GG  =-\omega_\conn(\dot \wuw_\PP ),\ \wuw_\GG(0)=e.
\label{diff4}
\end{equation}
Then  the identity
\begin{equation}
\QQQ\, \Hol_{\wuw,\QQQ}(\conn)=\wuw_\conn(1) =\wuw_\PP(1) 
\wuw_\GG(1)
\end{equation}
characterizes $\Hol_{\wuw,\QQQ}(\conn) \in \GG$.

The following observation
is a Fr\'echet version of~\cite[Theorem 2.1]{MR1600534}.

\begin{lem}
\label{horlift}
Given a smooth loop $\wuw\colon [0,1]\to \GGGG$ in 
$\GGGG$ having starting point $\prin(\QQQ)$,
the horizontal lift map 
\begin{equation}
 \mathrm{Hor}_{\wuw,\QQQ}\colon \Conn \longrightarrow 
\Gamma^\infty(\wuw^* \prin),\ \conn \longmapsto \wuw_\conn,
\end{equation}
with respect to $\QQQ$
is smooth
and, at a smooth connection $\conn$,  the derivative is given by
\begin{equation} \label{eq::gt:wilsonLoop:derivativeOfHorizontalLift}
\begin{aligned}
\tangent_\conn \mathrm{Hor}_{\wuw,\QQQ}&\colon 
\mathcal A^1(\GGGG, \adx) \longrightarrow C^\infty([0,1], \gg), 
\quad \vartheta \mapsto f_\vartheta,\\  
f_\vartheta(t)&=\int_0^t {\vartheta_{\wuw_\conn(\tau)}(\dot\wuw_\conn(\tau))} \mathrm{d}\tau,\ t \in [0,1].
\end{aligned}
\end{equation}
\end{lem}

To guide the reader through 
\eqref{eq::gt:wilsonLoop:derivativeOfHorizontalLift}, we note that,
for any $\tau \in [0,1]$, the notation
$\wuw_\conn(\tau)$ refers to a point of the total space $\PP$ of $\prin$
and
$\dot\wuw_\conn(\tau)$ to a tangent direction at $\wuw_\conn(\tau)$.

The idea that underlies the proof 
we are about to give
is the same as that for the proof of~\cite[Theorem 2.1]{MR1600534}. 
However, we here consider the horizontal lift as a map 
\( \Conn \to \Gamma^\infty(\wuw^* \prin) \) 
into the space of smooth sections of the induced bundle
$\wuw^* \prin$
rather than as a map of the kind
\( \Conn \times [0,1] \to \PP \). Moreover, we exploit 
a very natural and well-established notion of smoothness 
instead of the ad hoc concept of a map that is
\lq\lq smooth on every finite-dimensional submanifold\rq\rq\  
used in~\cite{MR1600534}.

\begin{proof}  
Let $\wuw\colon [0,1]\to \GGGG$ be a smooth loop in 
$\GGGG$ having starting point $\prin(\QQQ)$,
let $\conn$ be a smooth connection on $\prin$, and let
\( \wuw_{\conn} \) 
denote the horizontal lift of \( \wuw \) with respect to 
\( \conn \) and starting at \( \QQQ \).

The induced bundle $\wuw^*\prin$ trivializes.
We take the connection $\conn$ as reference connection, 
use the horizontal lift $\wuw_\conn$ 
of $\wuw$ relative to $\conn$
to trivialize
the bundle $\wuw^* \prin$ and
argue henceforth in terms of the
total space $[0,1] \times G$. 
The trivialization 
identifies the space $\Gamma^\infty(\wuw^* \prin)$ of smooth sections
of  $\wuw^* \prin$
with the space $C^{\infty}([0,1],G)$ of smooth $G$-valued maps on $[0,1]$.
By construction, in the chosen trivialization,
the original path $\wuw$ amounts to the identity path
$[0,1] \to [0,1]$.
Accordingly, given $\vartheta \in \form^1(\GGGG, \adx)$,
the component
$\wuw_{\conn +\vartheta} \colon [0,1] \to G$ into $G$
of the horizontal lift
$[0,1]\to [0,1] \times G$ of the identity path of $[0,1]$
determines the horizontal lift  of $\wuw$ relative to $\conn +\vartheta$.
Thus the horizontal lift map takes the form
\begin{equation} 
		\mathrm{Hor}_{\wuw,\QQQ}\colon 
\form^1(\GGGG, \adx) \longrightarrow C^\infty([0,1], \GG), 
\qquad \vartheta \mapsto \wuw_{\conn + \vartheta}.
\label{trivia}
\end{equation}
In particular, $\wuw_\conn$ is the trivial path at the neutral element
$e$ of $G$.

We use the variable $t$ as coordinate on $[0,1]$.
Let $\vartheta \in \form^1(\GGGG, \adx)$.
This $1$-form, restricted to
the bundle $\wuw^*\prin$,
now takes, along the factor $[0,1]$ of the decomposition
 $[0,1] \times G$ of the total space, 
 the form
$h_{\vartheta} dt$ for some smooth $\gg$-valued function $h_\vartheta$ 
on $[0,1]$ uniquely determined by $\vartheta$, 
and
the path \( \wuw_{\conn +\vartheta} \) is  the solution of the 
differential equation 
\begin{equation}
\dot\wuw_{\conn +\vartheta}\wuw_{\conn +\vartheta}^{-1} = h_\vartheta.
\label{deq}
\end{equation}
The sign here is consistent with 
\eqref{diff4}.
Indeed, in our realization of connections as vector bundle sections for
\eqref{ati1},
the connection   
 $\conn + \vartheta$ on $\prin$ has connection form
$\omega_\conn - \vartheta \colon \T \PP \to \gg$,
where we identify $\adx$-valued $1$-forms on $\GGGG$ with
$G$-equivariant $1$-forms $\T \PP \to \gg$ as usual.
The right-hand side of equation \eqref{deq} 
 depends continuously on the parameter 
\( \vartheta \in \form^1(\GGGG,\adx) \). 
The continuous dependence of solutions of differential equations on 
parameters~\cite[10.7.1]{MR0349288} 
now implies that 
the horizontal lift map \eqref{trivia}
is a continuous map.

Let $s$ denote a real variable.
Differentiating the identity
\begin{align}
\dot \wuw_{ \conn + s\vartheta}\wuw_{ \conn + s\vartheta}^{-1}&=s h_\vartheta 
\label{id21}
\end{align}
with respect to $s$ yields the identity
\begin{align*}
\tfrac{\partial}{\partial s}
(\dot \wuw_{ \conn + s\vartheta})\wuw_{ \conn + s\vartheta}^{-1}
+\dot \wuw_{ \conn + s\vartheta}
\tfrac{\partial}{\partial s}
(\wuw_{ \conn + s\vartheta}^{-1})
&=h_\vartheta .
\end{align*}
However, the path $\wuw_\conn$
is constant whence, for $s=0$,  the differential equation reduces to
\begin{align*}
\tfrac{\partial}{\partial s}\big|_{s=0}
(\dot \wuw_{ \conn + s\vartheta})
&=h_\vartheta .
\end{align*}
Interchanging the order of differentiation, we write this equation as
\begin{align*}
\tfrac{d}{dt}
\tfrac{\partial}{\partial s}\big|_{s=0}
(\wuw_{ \conn + s\vartheta})
&=h_\vartheta .
\end{align*}
Consequently
\begin{align*}
\tfrac{\partial}{\partial s}\big|_{s=0}
(\wuw_{ \conn + s\vartheta})(t)
&=\int_0^t h_\vartheta(\tau) d\tau 
=\int_0^t {\vartheta_{\wuw_\conn(\tau)}(\dot\wuw_\conn(\tau))} \mathrm{d}\tau.
\end{align*}

The associated {\em Magnus series\/} provides more insight:
Since \eqref{trivia}
is a continuous map,
the horizontal lift with respect to a connection near 
\( \conn \) lies in an appropriate tubular neighborhood
of  \( \wuw_\conn \). 
Hence, when $\vartheta$
is sufficiently close to zero, we
can
write the path $\wuw_{\conn +\vartheta}$ as
\begin{equation*}
\wuw_{\conn +\vartheta} = \mathrm{exp}(\mathrm W_{\conn +\vartheta})
\end{equation*}
for some path
$\mathrm W_{\conn +\vartheta}\colon [0,1] \to \gg$. 
Using the familiar identity
\begin{align*}
\tfrac{d}{dt}(\mathrm{exp}(Y(t))) \mathrm{exp}(Y(t))^{-1}&=
\frac { 
\mathrm e^{\mathrm{ad}(Y(t))}-1}{\mathrm{ad}(Y(t))}
\dot Y(t),
\end{align*}
letting $Y=\mathrm W_{\conn +\vartheta}$, we 
rewrite equation \eqref{deq}
as
\begin{align}
\frac { 
\mathrm e^{\mathrm{ad}(Y)}-1}{\mathrm{ad}(Y)}
\dot Y
&=h_\vartheta .
\label{deqq}
\end{align}
The first two terms
of a {\em Magnus expansion\/} 
$Y(t) =\sum_{j \geq 1} Y^{(j)} (t)$
of the solution \( Y \) of \eqref{deqq}
read
\begin{align}
		Y^{(1)}(t) &= 
\int_0^t \mathrm{d} \tau\ h_\vartheta(\tau), \\
	 	Y^{(2)}(t) &= 
                  -\frac{1}{2} \int_0^t \mathrm{d} \tau_1 
       \int_0^{\tau_1} \mathrm{d} \tau_2\ 
[h_\vartheta(\tau_1), h_\vartheta(\tau_2)],  
\end{align}
and the higher order terms involve higher order commutators of 
\( h_\vartheta \) 
evaluated at different times. The expansion yields explicit 
expressions for the derivatives of the
horizontal lift map \eqref{trivia},
since every order in the series equates to the order of differentiation. 
In particular, this argument 
confirms that \eqref{trivia} 
is smooth and that the
expression~\eqref{eq::gt:wilsonLoop:derivativeOfHorizontalLift}
yields the derivative.
\end{proof}

We return to our surface $\Sigma$ and
maintain
the choice of base point $Q$ and of
a canonical system $u_1,v_1,\dots,u_\ell,v_\ell$ of 
closed curves in \( \Sigma \) whose
homotopy classes generate \( \pi_1(\Sigma, Q) \), cf. \cref{univex}.
Consider a principal $G$-bundle $\prin \colon \PP \to \Sigma$ on $\Sigma$
and choose a base point
$\QQQ$ of $\PP$ over  the base point $Q$ of $\Sigma$.
Henceforth the base point $\QQQ$ of $\PP$ remains fixed.
Given a smooth connection $\conn$ on $\prin$,
for $1 \leq j \leq \ell$,
let $u_{\conn,j}$ and $v_{\conn,j}$ 
denote the horizontal lifts, in the total space $P$ of $\prin$, of
$u_j$ and $v_j$, respectively,
with reference to $\conn$ and $\QQQ$.
The \emph{Wilson loop mapping}
\begin{equation}
\wilson \colon \Conn  \longrightarrow G^{2\ell}
\label{wilson1}
\end{equation}
relative  to $\QQQ$ and $u_1,v_1,\dots,u_\ell,v_\ell$ 
assigns to a smooth connection $\conn$ on $\prin$ the point
\begin{equation}
\wilson(\conn) =
\left(
\mathrm{Hol}_{u_1,\QQQ}(\conn),
\mathrm{Hol}_{v_1,\QQQ}(\conn),
\dots,
\mathrm{Hol}_{u_\ell,\QQQ}(\conn),
\mathrm{Hol}_{v_\ell,\QQQ}(\conn)
\right)
\label{value}
\end{equation}
of $G^{2\ell}$~\cite[(2.6)]{MR1600534}.
This map depends on the choices 
made to carry out its construction.
In \cref{dep} below we  explore the dependence on 
these choices.

For $q \in G$, we denote the induced operation of left translation
by
\[
L_q\colon \gg=\mathrm T_eG \longrightarrow \mathrm T_q G.
\]
The subsequent result
is the Fr\'echet version of~\cite[Theorem 2.7]{MR1600534},
spelled out for the special case
where the base manifold is a closed surface.

\begin{thm}\label{adapted}
In the Fr\'echet topologies, the Wilson loop mapping \( \wilson \colon \Conn  \to G^{2\ell}  \), cf. \eqref{wilson1} above,
is smooth.
At a smooth connection $\conn$, the value 
{\rm \eqref{value}}
under \( \wilson \) being written as
\begin{equation*}
\begin{aligned}
\wilson(\conn) &=
(a_1,b_1,\ldots,a_{\ell},b_{\ell})\in G^{2\ell},
\end{aligned}
\end{equation*}
the tangent map
\begin{equation}
\tangent_\conn \wilson
\colon
\mathrm T_\conn\Conn
\longrightarrow
\mathrm T_{\wilson(\conn)} G^{2\ell}
=
\mathrm T_{a_1} G
\times
\mathrm T_{b_1} G
\times
\dots
\times
\mathrm T_{a_\ell} G
\times
\mathrm T_{b_\ell} G
\label{diff2}
\end{equation}
of \( \wilson \) sends
$\vartheta \in \form^1(\Sigma,\adx)= \mathrm T_\conn\Conn$
to
\begin{equation}
\left(
L_{a_1} \int_{u_{\conn,1}} \vartheta,
L_{b_1} \int_{v_{\conn,1}} \vartheta,
\dots,
L_{a_\ell} \int_{u_{\conn,\ell}} \vartheta,
L_{b_\ell} \int_{v_{\conn,\ell}} \vartheta
\right),
\end{equation}
by construction a vector in
$\mathrm T_{a_1} G
\times
\mathrm T_{b_1} G
\times
\dots
\times
\mathrm T_{a_\ell} G
\times
\mathrm T_{b_\ell} G.
$
\end{thm}

\begin{proof}
This is a consequence of Lemma \ref{horlift}
since, given a smooth closed path $\wuw\colon [0,1]\to \Sigma$
starting at the base point $Q$,
evaluation at \( 1 \) is a smooth map 
\( \Gamma^\infty(\wuw^* \prin) \to P_Q \simeq G \), 
and the derivative thereof also comes down to evaluation at \( 1 \).
\end{proof}

\section{Preparing for the description of the Fr\'echet constraints}
\label{preparing}

To explain the idea as to how the constraints
arise,
consider a 2-form $\beta$ on $\Sigma$ such that,
with respect to the volume form $\vol_{\Sigma}$
on $\Sigma$ chosen in \cref{univex} above,
$\int_\DDD \beta= \int_\DDD \vol_{\Sigma}$
for any disk $\DDD$ in $\Sigma$.
Then $\int_\DDD (\beta-\vol_{\Sigma})=0$,
for any disk $\DDD$ in $\Sigma$ whence
\[
\beta=\vol_{\Sigma} + \dif \alpha
\]
for some 1-form $\alpha$ on $\Sigma$.
Now 
\[
\int_{\partial \DDD} \alpha = \int_\DDD \dif \alpha 
= \int_\DDD(\beta-\vol_{\Sigma}) = 0,
\]
for any disk $\DDD$ in $\Sigma$. However, a \( 1 \)-form whose integral over 
any closed
contractible path vanishes is an exact form.
Hence \( \dif \alpha = 0 \) and so $\beta =\vol_{\Sigma}$.
In other words, the functionals 
given by integration over arbitrary disks separate points of \( \mathcal A^2(\Sigma, \mathbb{R}) \).

We return to our principal $G$-bundle $\prin \colon \PP \to \Sigma$ on 
$\Sigma$ and maintain the choice of  base point 
$\QQQ$ of $\PP$ over the base point $Q$ of $\Sigma$.
For convenience, we slightly vary the classical construction
that reduces a principal bundle with connection
to the holonomy bundle at a chosen base point of the total space,
cf., e.~g., the reduction theorem
\cite[II.7.1]{MR0152974}.
Recall the choice of 
 a principal $\mathrm U(1)$-connection
form $\omega_{\MS}\colon \T \MS \to i \mathbb R$ 
on $\prin_{\MS}\colon \MS \to \Sigma$
and its lift
 $\omega_{\wMS}\colon \T \wMS \to  \mathbb R$ 
to a principal $\GR$-connection form on the
principal $\GR$-bundle
$\prin_{\wMS}\colon\wMS \to \Sigma$ 
and, furthermore,
the choice of
a base point
$Q_{\wMS}$ of $\wMS$
over $Q$.
Let $\conn$ be a smooth connection on $\prin$, and let
$\PP_{\mathrm{hor}}$ denote the space of paths 
$\wuw\colon [0,1] \to \wMS$ in $ \wMS$ that start at
$Q_{{\wMS}}$ and are horizontal relative to $\conn$, and let
$\mathrm{ev}\colon \PP_{\mathrm{hor}} \to \wMS $
denote the evaluation map which sends a path in 
$\PP_{\mathrm{hor}}$ to its end point.
This map is surjective.
Indeed,
since the principal $\GR$-bundle 
$\prin_{\wMS}\colon\wMS \to \Sigma$
coincides with its holonomy bundle at
$Q_{\wMS}$
 with respect to the connection form
$\omega_{\wMS}$, any point $T$ of $\wMS$
is the end point of a smooth horizontal path joining 
$Q_{\wMS}$ to $T$.

We define the {\em pre-reduction map\/}
$\widehat\lift_{Q_{{\wMS}},\QQQ, \conn}\colon \PP_{\mathrm{hor}} \to \PP$
{\em of\/} $\conn$
{\em relative to   
$Q_{{\wMS}}$ and\/} $\QQQ$ as follows:
Given a member
$\widetilde \wuw\colon [0,1]\to \wMS$ 
of $\PP_{\mathrm{hor}}$, that is,
a path
in
${\wMS}$ that is horizontal
relative to  $\omega_{\wMS}$ and
starts at the point
$Q_{{\wMS}}$ of $\wMS$,
let $\wuw$ be the path in $\Sigma$ obtained by projecting
$\widetilde \wuw$ into $\Sigma$,
and let
$\wuw_\PP\colon [0,1]\to \PP$
be the unique  lift
of $\wuw$
that is horizontal
for $\conn$ and has starting point
$\QQQ$;
define the value $\widehat \lift_{Q_{{\wMS}},\QQQ, \conn}(\widetilde u)$ 
to be the end point
of
$\wuw_\PP$.
We say that the pre-reduction map {\em is defined on\/}
$\wMS$ when it factors through
the evaluation map $\mathrm{ev}\colon \PP_{\mathrm{hor}} \to \wMS$;
we then write the resulting map as
$\lift_{Q_{{\wMS}},\QQQ, \conn}\colon {\wMS} \to P$.

Given a disk $\DDD$ in $\Sigma$, let $a_\DDD=\int_\DDD\vol_\Sigma$ 
denote the area of $\DDD$
with respect to $\vol_\Sigma$.
Recall the topological 
characteristic class  \( X_\prin \in \hh \) of \( \prin \), see \cref{top1}.

\begin{lem}
\label{crucial} Consider a smooth connection $\conn$ on 
$\prin\colon P\to \Sigma$.
The following are equivalent:

\begin{enumerate}
	\item
		The values of the curvature form
		$\mathrm{curv}_{\conn}$ of $\conn$
		lie in the center $\zzz$ of the Lie algebra $\gg$ of $\GG$
		in such a way that $\curv_\conn= -X_\prin\cdot \vol_\Sigma$.
		
	\item
		The values of the curvature form 
		$\mathrm{curv}_{\conn}$ of $\conn$
		lie in the center $\zzz$ of the Lie algebra $\gg$ of $\GG$,
		and
		the curvature form $\curv_\conn$ of $\conn$, being accordingly viewed
		as an ordinary  $\zzz$-valued $2$-form on $\Sigma$,
		satisfies the identity
		\begin{equation}
		\int_\DDD\mathrm{curv}_{\conn} =  -a_\DDD X_{\prin}
		\left (= -\int_\DDD \vol_{\Sigma} \,X_{\prin}\right) \in \zzz
		\label{constraint1}
		\end{equation}
		whenever $\DDD \subseteq \Sigma$ is a disk in $\Sigma$.
	\item
		For every smooth contractible loop $\wuw$ in $\Sigma$ 
		starting at some point $q$ of $\Sigma$
		and bounded by a disk $\DDD$
		in $\Sigma$, 
		\begin{equation}
		\mathrm{Hol}_{\wuw,q_\PP} (\conn) =\exp (a_\DDD X_{\prin})\in \ZZZ,
		\label{constraint11}
		\end{equation}
		for any choice of pre-image $q_\PP \in \PP$ of $q\in \Sigma$.
	
	\item
		The pre-reduction map  of $\conn$
		 relative to   
		$Q_{{\wMS}}$ and $\QQQ$ is defined on $\wMS$ and
		yields a morphism
		\begin{equation}
		\begin{CD}
		\GR
		@>>>
		{\wMS}
		@>{\prin_{{\wMS}}}>>
		\Sigma
		\\
		@V{\cchi_{\conn}}VV
		@V{\lift_{Q_{{\wMS}},\QQQ, \conn}}VV
		@|
		\\
		G
		@>>>
		\PP
		@>{\prin}>>
		\Sigma
		\end{CD}
		\label{redu1}
		\end{equation}
		of principal bundles with connection in such a way that
		the following hold:

		\begin{enumerate}
			\item[{\rm(a)}]
				The homomorphism
				$\wilson(\conn)\colon F \to G$
				which arises as the value of $\conn$ under the Wilson loop mapping
				{\rm \eqref{wilson1}}
				coincides with the composite
				\begin{equation*}
				F \longrightarrow \Gamma \longrightarrow \GR
				\stackrel{\cchi_{\conn}} \longrightarrow G,
				\end{equation*}
				the unlabeled arrows being the obvious maps.
				
			\item[{\rm (b)}]
				The homomorphism $\cchi_{\conn}$ satisfies the identity
				$\cchi_{\conn}(t[r])=\mathrm{exp}(t X_{\prin})$, for $t \in \R$.
		\end{enumerate}
\end{enumerate}
\end{lem}

\begin{rema} Since the genus $\ell$ of $\Sigma$ is 
(supposed to be)
positive, 
given a contractible loop 
$\wuw$ in $\Sigma$ bounded by a disk, that disk is uniquely determined
by $\wuw$ (up to reparametrization).
In this sense, the right-hand side of~\eqref{constraint11} depends
only on $\wuw$.
\end{rema}

\begin{proof}[Proof of Lemma {\rm \ref{crucial}}]
In view of the observation spelled out at the beginning of the present 
section,
the equivalence of (i) and (ii) is immediate. It is also immediate that
(ii) implies (iii).
We now show that (iii) implies (i).

When $G$ is semisimple,
 the topological characteristic class $X_{\prin}$ of the bundle
$\prin$
is zero,
and the constraints \eqref{constraint11} 
simplify to
\begin{equation}
\mathrm{Hol}_{\wuw,q_\PP}(A)=\{e\}\ \text{whenever $\wuw$ is a contractible loop in}\ \Sigma,
\end{equation}
for a choice  $q\in \Sigma$ of starting point of $\wuw$ and of
a pre-image $q_\PP \in \PP$ of $q$.

Consider now the case where the
center of $\GG$ has dimension $\geq 1$ and,
as before, let $\zzz$ denote 
the center of the Lie algebra $\gg$ of $\GG$. 
The group   
$\GG/\ZZZ= \overline S\cong S/D$ is semisimple, 
and the bundle projection $\prin\colon \PP \to \Sigma$
factors through the induced map from
the orbit manifold $\PP/\ZZZ$ to $\Sigma$, necessarily
a principal
 $\overline S$-bundle, and we denote this bundle by
$\prin_{\overline S}\colon \PP/\ZZZ \to \Sigma$.

Consider a smooth $\GG$-connection $\conn$ 
on $\prin$ that
satisfies the 
constraints~\eqref{constraint11},
and let $\conn_{\overline S}$ denote the 
$\overline S$-connection
it induces on  $\prin_{\overline S}$.
Relative to 
$\prin_{\overline S}$,
the corresponding constraints~\eqref{constraint11} say that,
given a smooth contractible loop $\wuw$ in $\Sigma$, 
for a choice $q\in \Sigma$
of starting point  of $\wuw$ and
of a pre-image $q_\ZZZ \in \PP/\ZZZ$ of $q\in \Sigma$,
\begin{equation*}
\Hol_{\wuw, q_\ZZZ}(\conn_{\overline S})=e \in  \overline S. 
\end{equation*}
Hence the connection $\conn_{\overline S}$ on $\prin_{\overline S}$ is flat.
Consequently the curvature form
$\mathrm{curv}_{\conn} \in \mathcal A^2(\Sigma,\adx)$ 
of $\conn$
is 
{\em central\/} 
in the sense that
the values of $\mathrm{curv}_{\conn}$ lie in
the center $\zzz$ of $\gg$.

Relative to the decomposition $\gg = \hh \oplus \mathfrak s$
of the Lie algebra $\gg$ into its center $\hh$ and its semisimple
constituent $\mathfrak s=[\gg,\gg]$, the connection form $\omega_\conn$
of $\conn$
decomposes as
$\omega_\conn = \omega_\hh +\omega_{\mathfrak s}$ and, accordingly,
the curvature form $\curv_\conn$ of $\conn$ decomposes as
\begin{equation}
\curv_\conn = \dif \omega_\hh + 
\dif \omega_{\mathfrak s} + \tfrac 12[\omega_{\mathfrak s},\omega_{\mathfrak s}].
\end{equation}
Since the values of the curvature form $\curv_\conn$ lie in the center 
$\hh$ of $\gg$,
the
constituent
$d\omega_{\mathfrak s} + \tfrac 12[\omega_{\mathfrak s},\omega_{\mathfrak s}]$
is zero. 

Consider a contractible loop $\wuw\colon [0,1]\to \Sigma$ in $\Sigma$ 
and a disk $\DDD$ in $\Sigma$ that bounds $\wuw$.
Choose  a starting point $q$ of $\wuw$ in $\Sigma$
with $w(0)=q$ and a pre-image $q_\PP \in \PP$
of $q$. Since the loop $\wuw$ is contractible,
the summand $\omega_{\mathfrak s}$ 
of the decomposition 
$\omega_\conn = \omega_\hh +\omega_{\mathfrak s}$ 
does not contribute
to the $\GG$-holonomy of $\wuw$ relative to $\conn$
and $q_\PP$.

To justify this claim,
consider the restriction of
$\prin$ to the disk $\DDD$.
This restriction trivializes, and we fix a smooth $\GG$-equivariant
embedding
$\DDD \times \GG \to \PP$ that covers the embedding $\DDD \to \Sigma$
in such a way that
the embedding sends
the point $(q,e)$ of $\DDD \times \GG$ to the point $q_\PP$ of $\PP$.
Consider the restrictions
$\eta_\hh\in \form^1(\DDD, \hh)$ and 
$\eta_{\mathfrak s}\in \form^1(\DDD, \mathfrak s)$, to $\DDD$, of the 
respective $1$-forms 
$\omega_\hh$  and $\eta_{\mathfrak s}$.
The restriction, to $\DDD$,
of the connection form $\omega_\conn$ of $\conn$
now
coincides with the sum 
$\eta_\conn = \eta_\hh + \eta_{\mathfrak s}\in \form^1(\DDD, \gg)$.
This sum yields a connection form
relative to the covering group $\ZZZ \times S$ of $\GG= H \times_D S$. 

The closed path $\wuw$ now amounts to the boundary
path $\wuw\colon [0,1] \to \DDD$ of $\DDD$ having starting point $q \in \DDD$,
and the holonomy with respect to $\conn$ and $(q_\PP,e)$,
in $\ZZZ \times S$,
 of $\wuw$ 
is the end point of the path
$(\wuw_\ZZZ,\wuw_S)\colon [0,1] \to \ZZZ \times S$
that solves the two differential equations
\begin{alignat*}{2}
\dot \wuw_\ZZZ\wuw_\ZZZ^{-1}&= -\eta_\hh(\dot \wuw),\  
\wuw_\ZZZ(0)&&=e \in \ZZZ,
\\
\dot \wuw_S \wuw_S^{-1}&=-\eta_{\mathfrak s}(\dot \wuw),\  
\wuw_S(0)&&=e \in S.
\end{alignat*}
Since $\eta_{\mathfrak s}$ arises from a flat connection
and since the path $\wuw$ is contractible,
the path $\wuw_S$ is closed.
Consequently that holonomy,
in $\ZZZ \times S$,
 of $\wuw$
is the end point of the path
$(\wuw_\ZZZ,e)\colon [0,1] \to \ZZZ \times S$.

As in \cref{wilson}, let
$\wuw_\PP\colon [0,1] \to \PP$
denote a closed lift of $\wuw$
having the point $q_\PP$ as its starting point.
We can take the composite of 
the boundary path
$\wuw\colon [0,1]\to \DDD$
with the embedding $\DDD \to \DDD\times \GG \to \PP$ but this is not strictly
necessary.
Since  the summand \( \omega_{\mathfrak s} \) of
$\omega_\conn = \omega_\hh +\omega_{\mathfrak s}$ 
does not contribute
to the $\GG$-holonomy of $\wuw$ relative to $\conn$
and $q_\PP$,
in view of \eqref{diff4},
\begin{equation}
\mathrm{Hol}_{\wuw,q_\PP} (\conn)=
\mathrm{exp}\left(-\int_{\wuw_\PP}\omega_\hh\right).
\end{equation}
Since
\begin{equation}
\mathrm{exp}\left(-\int_{\wuw_\PP}\omega_\hh\right)
=
\mathrm{exp}\left(-\int_{\DDD}\dif\omega_\hh\right)
=
\mathrm{exp}\left(
-\int_\DDD\curv_\conn\right),
\end{equation}
the constraint
\eqref{constraint11}
implies that
\begin{equation}
\mathrm{exp}\left(
-\int_\DDD\curv_\conn\right)
= \exp (a_\DDD X_{\prin}) 
\in \ZZZ \ \text{whenever}\ \DDD\ \text{is a disk in}\ \Sigma.
\end{equation}

Consider the $\zzz$-valued 2-form 
$\sigma=\curv_\conn + X_\prin\cdot \vol_\Sigma$
on $\Sigma$.
This 2-form has the property that
$\exp\left(\int_\DDD \sigma\right)=e \in \ZZZ$, for any disk $\DDD$ in 
$\Sigma$.
Consequently $\int_\DDD \sigma\in \pi_1(\ZZZ) \subseteq \zzz$,
for any disk $\DDD$ in 
$\Sigma$. 
Now fix a disk $\DDD$.
Within $\DDD$, consider a descending sequence
$\{ \DDD_\varepsilon \}_\varepsilon$ of disks whose intersection is a single point.
In a suitable coordinate patch, we realize the disks $\DDD_\varepsilon$ as ordinary disks
having radius $\varepsilon \in \mathbb R$.
Now, on the one hand,
the map $f\colon I \to \pi_1(\ZZZ)$ 
given by $f(\varepsilon)=\int_{\DDD_\varepsilon} \sigma$
is continuous and hence constant, and on the other hand,
$\lim_{\varepsilon \to 0}f(\varepsilon)=\lim_{\varepsilon \to 0}\int_{\DDD_\varepsilon} \sigma =0$,
whence $\int_\DDD \sigma=0$.
Since $\DDD$ is an arbitrary disk, we conclude $\sigma=0$, that is,
$\curv_\conn = -X_\prin\cdot \vol_\Sigma$ as asserted.

We now show that (i) implies (iv).
Consider first the special case where the group $\GG$ is abelian, i.e., 
in our notation,
coincides with the compact connected abelian subgroup $\ZZZ$ of $\GG$.
Thus $\prin\colon \PP \to \Sigma$
is momentarily a principal $\ZZZ$-bundle with an $\ZZZ$-connection
$\conn$ having curvature $\curv_\conn = -X_\prin \cdot \vol_\Sigma$.

Recall from \cref{top1} above that the characteristic class $X_\prin$ 
determines the corresponding homomorphism
$\chi_\prin\colon \mathrm U(1) \to \ZZZ$
via the identity ${\chi_\prin(\mathrm e^{i t})=\mathrm{exp}(t X_\prin)}$ 
($t \in [0,2 \pi]$).
Let $\prin_\conn \colon \PP_\conn \to \Sigma$
denote the principal $\ZZZ$-bundle
associated to our reference $\mathrm U(1)$-bundle 
$\prin_{\MS}\colon \MS \to \Sigma$ via $\chi_\prin$,
and denote the resulting smooth map from $\MS$ to $\PP_\conn$
by $\lambda$.
The commutative diagram
\begin{equation}
\begin{CD}
\mathrm U(1)
@>>>
\MS
@>{\prin_{\MS}}>>
\Sigma
\\
@V{\cchi_\prin}VV
@V{\lift}VV
@|
\\
\ZZZ
@>>>
\PP_\conn
@>{\prin_\conn}>>
\Sigma 
\end{CD}
\label{redu12}
\end{equation}
displays the resulting bundle map.
Our reference connection form 
${\omega_{\MS}\colon \T {\MS} \to i \mathbb R}$
on $\prin_{\MS}\colon \MS \to \Sigma$ induces
an $\ZZZ$-connection form on $\prin_\conn$ having curvature form 
$-X_\prin \cdot \vol_\Sigma$,
and we denote this connection form by 
$(\cchi_\prin)_*\omega_{\MS}\colon \T \PP_\conn \to i \mathbb R$.
The bundles $\prin$ and $\prin_\conn$ are topologically equivalent.
Hence there is an $\ZZZ$-equivariant diffeomorphism 
$\Phi\colon \PP \to \PP_\conn$ covering the identity of $\Sigma$
that induces an isomorphism $\prin_\conn \to \prin$ of
principal $\ZZZ$-bundles.
The $\ZZZ$-connection form $(\cchi_\prin)_*\omega_{\MS}$
induces, via $\Phi$,
an $\ZZZ$-connection form $\omega_{\prin,\Phi}\colon \T \PP \to i \mathbb R$
on $\prin$ having, likewise, curvature form $-X_\prin \cdot \vol_\Sigma$.
Since this is the curvature form of $\conn$ as well,
for a suitable bundle automorphism (gauge transformation)
$\beta\colon \PP \to \PP$, 
pulling back 
the connection form
$\omega_\conn$ of $\conn$
via $\beta$ yields $\omega_{\prin,\Phi}$ as 
$\omega_{\prin,\Phi}=\beta^*\omega_\conn$.
By construction, $\lift$ has the form of a pre-reduction
map, now defined on
$\MS$ rather than on $\wMS$,
 and
$\Phi$ and $\beta$ send horizontal paths to horizontal paths.
The composite $\beta\circ \Phi \circ \lambda$ yields
the 
morphism
\begin{equation}
\begin{CD}
\mathrm U(1)
@>>>
\MS
@>{\prin_{\MS}}>>
\Sigma
\\
@V{\cchi_\prin}VV
@V{\beta\circ \Phi \circ\lift}VV
@|
\\
\ZZZ
@>>>
\PP
@>{\prin}>>
\Sigma 
\end{CD}
\label{redu13}
\end{equation}
of principal bundles with connection.
Claim (iv) now boils down to the observation that  
the composite of 
$\beta\circ \Phi \circ \lambda$
with the universal covering projection
$\wMS \to \MS$
admits the description
spelled out as pre-reduction map.

Now we consider the case of a general principal $\GG$-bundle
$\prin \colon \PP \to \Sigma$ with a $\GG$-connection $\conn$
whose curvature form reads $\curv_\conn = -X_\prin \cdot \vol_\Sigma$.
Maintain the notation $\overline \ZZZ = \GG/S \cong \ZZZ/D$ 
and
$\prin_{\overline \ZZZ}\colon \PP/S \to \Sigma$
introduced in \cref{top1} above and let $\overline \conn$ denote the
induced $\overline \ZZZ$-connection on 
the principal $\overline \ZZZ$-bundle $\prin_{\overline \ZZZ}$.
Let $Q_{\PP/S}$ denote the image of $\QQQ$ in $\PP/S$.
The above argument shows that,
with a slight abuse of the notation
$\lift_{Q_{\MS},Q_{\PP/S}, \overline \conn}$,
the $\mathrm U(1)$-bundle $\prin_{\MS}\colon \MS \to \Sigma$
being endowed with the reference connection form 
${\omega_{\MS}\colon \T {\MS} \to i \mathbb R}$,
 the expression for the
pre-reduction map of $\overline \conn$ relative to
$Q_{{\MS}}$ and $Q_{\PP/S}$ yields a morphism
\begin{equation}
\begin{CD}
\mathrm U(1)
@>>>
{\MS}
@>{\prin_{{\MS}}}>>
\Sigma
\\
@V{\cchi_{\prin_{\overline \ZZZ}}}VV
@V{\lift_{Q_{\MS},Q_{\PP/S}, \overline \conn}}VV
@|
\\
\overline H
@>>>
\PP/S
@>{\prin_{\overline \ZZZ}}>>
\Sigma
\end{CD}
\label{redu14}
\end{equation}
of principal bundles with connection.
Since
 the summand $\omega_{\mathfrak s}$ in the decomposition
$\omega_\conn =\omega_\hh +\omega_{\mathfrak s}$
of the connection form $\omega_\conn$ of $\conn$
is the connection form of 
a flat $S$-connection $\conn_{\mathfrak s}$ on the principal $S$-bundle $\PP \to \PP/S$, 
the expression for the
pre-reduction map of $ \conn$ relative to
$Q_{{\wMS}}$ and $\QQQ$ yields the unique lift 
$\widetilde M_{\Sigma}\to \PP$ of
$\lift_{Q_{{\MS}},Q_{\PP/S}, \overline \conn}\colon 
M_{\Sigma}\to \PP/S$ with respect to the flat connection 
$\conn_{\mathfrak s}$
and the choice
of base points $Q_{{\wMS}}$ of ${\wMS}$ and $\QQQ$ of $\PP$. 
The commutative diagram
\begin{equation}
\begin{CD}
\widetilde M_{\Sigma}
@>{\lift_{Q_{{\wMS}},\QQQ, \conn}}>> \PP
\\
@VVV
@VVV
\\
M_{\Sigma}@>{\lift_{Q_{{\MS}},Q_{\PP/S}, \overline \conn}}>> \PP/S
\end{CD}
\end{equation}
displays the situation.

Let $\wuw_r$ denote the \lq\lq boundary path\rq\rq\ 
$\prod u_j v_j u^{-1}_j v_j^{-1}$ of the 2-cell $e^2$ of $\Sigma$;
its horizontal lift
$\widetilde \wuw_r$  in ${\wMS}$
joins $Q_{{\wMS}}$ to $Q_{{\wMS}} [r] \in {\wMS}$.
Let
$\wuw_{\PP,r}$
denote the horizontal lift in $P$ of  $\wuw_r$  
relative to $\conn$ having starting point $Q_\PP$.
Then
$Q_\PP\,\Hol_{\wuw_r,\QQQ}(\conn)$ coincides with the end point $\wuw_{\PP,r}(1)$ 
of $\wuw_{\PP,r}$.
 By construction,
\begin{equation}
\QQQ\,\cchi_{\conn}([r])
= \lift_{Q_{{\wMS}},\QQQ, \conn}(Q_{{\wMS}} [r])
=\wuw_{\PP,r}(1) 
=\QQQ\, \Hol_{\wuw_r,\QQQ}(\conn).
\end{equation}
By (iii), 
$\Hol_{\wuw_r,\QQQ}(\conn)=\mathrm{exp}(X_\prin)$ 
whence
$\cchi_{\conn}([r])= \mathrm{exp}(X_\prin)$. 
Since the homomorphism
 $\cchi_{\conn}$  and the surjective homomorphism
$\GR \to \mathrm U(1)$ 
associated with the universal covering projection
$\wMS \to \MS$
render the diagram
\begin{equation*}
\begin{CD}
\GR @>{\cchi_{\conn}}>> \GG
\\
@VVV
@VVV
\\
\mathrm U(1)
@>{\cchi_{\prin_{\overline \ZZZ}}}>>
\overline \ZZZ
\end{CD}
\end{equation*}
commutative, the homomorphism $\cchi_{\conn}$ satisfies the identity
$\cchi_{\conn}(t[r])=\mathrm{exp}(tX_\prin)$, for $t \in \R$.

It is immediate that (iv) implies (i).
This completes the proof.
\end{proof}

\section{The constraints to be imposed on the based gauge orbit manifold}
\label{constraints}

In the Atiyah--Bott approach, the moduli space arises by symplectic reduction
with respect to the group of gauge transformations,
the vector space $\mathcal A^2(\Sigma,\adx)$
being viewed as the dual of the Lie algebra of infinitesimal 
gauge transformations,
and the map which assigns to a connection its curvature
being interpreted as a momentum mapping.
In the Fr\'echet setting, the dual
of an infinite-dimensional vector space 
 is a delicate notion.
Below we substitute for that momentum mapping  a suitable infinite family of
 smooth $\GG$-equivariant
$\GG$-valued maps on  the orbit manifold
$\mathcal B_{\prin,Q}=\GauQ \backslash \Conn $,
and we refer to these maps as \emph{constraints}. 
The Wilson loop mapping \( \rho \), cf. \eqref{wilson1},
is invariant under the action of the group of based gauge transformations and hence
induces a smooth map
\begin{equation}
\wilsonn \colon \mathcal B_{\prin,Q} \longrightarrow \GG^{2\ell}
\label{wilson2}
\end{equation}
between Fr\'echet manifolds.
With a slight abuse of terminology, we refer to
$\wilsonn$ as {\em Wilson loop mapping\/} as well.
Via evaluation,
the choice of generators
for the presentation \eqref{3.1} of $\pi_1(\Sigma,Q)$
induces an obvious
diffeomorphism
from  $\Hom(F,\GG)$ onto $\GG^{2\ell}$,
and we identify the two manifolds.
The constraints will 
single out a subspace such that the restriction
of the Wilson loop mapping \( \wilsonn \) to that subspace is a 
homeomorphism
 onto
the subspace
$\Hom_{X_\prin}(\Gamma,\GG)_\prin $ of $\Hom(F,\GG)\cong \GG^{2\ell}$
introduced at the end of \cref{top1} above.
The subspace 
of $\mathcal B_{\prin,Q}$ 
 thus cut out by the constraints will recover the space
of  based gauge equivalence classes of smooth central 
Yang-Mills connections on $\prin$.

For a smooth connection $\conn$ on 
$\prin\colon \PP \to \Sigma$, we denote  by $[\conn]$ its class
in the orbit manifold
$\mathcal B_{\prin,Q}=\GauQ \backslash \Conn $.
Consider the maps
\begin{equation}
c_{\wuw,\DDD,\QQQ}\colon\mathcal B_{\prin,Q} \longrightarrow \GG,
\ 
[\conn]\longmapsto 
\mathrm{Hol}_{\wuw,\QQQ} (\conn) \exp (a_\DDD X_{\prin}),
\end{equation}
as $\wuw$ ranges over contractible loops
in $\Sigma$ having the base point $Q$ of $\Sigma$
as starting point and
bounded by disks $\DDD$ in $\Sigma$.
Lemma \ref{horlift} entails that
these maps are smooth.
Let
\[
\mathcal {CYM}_{\prin,Q}\subseteq \mathcal B_{\prin,Q}
\]
be the subspace of $\mathcal B_{\prin,Q}$ cut out by the constraints
\begin{equation}
\label{eq::constraints}
c_{\wuw,\DDD,\QQQ}([\conn])=e \in \GG.
\end{equation}
\cref{crucial} entails that \( \mathcal {CYM}_{\prin,Q} \) 
recovers
 the space of smooth based gauge equivalence classes
of smooth central Yang--Mills connections on $\prin$
and that the values of
the restriction of the Wilson loop mapping
{\rm \eqref{wilson2}}
to $\mathcal {CYM}_{\prin,Q}$
lie in
$\Hom_{X_\prin}(\Gamma,\GG)_\prin$.
We endow
$\mathcal {CYM}_{\prin,Q}$ 
with the topology induced from 
 the Fr\'echet topology of $\mathcal B_{\prin,Q}$
and
 refer to this topology as the {\em Fr\'echet topology\/} on 
$\mathcal {CYM}_{\prin,Q}$.

\begin{thm}
\label{singleout}
The space $\mathcal {CYM}_{\prin,Q}$ is non-empty, 
the values of the
restriction
of the Wilson loop mapping 
\( \wilsonn \), cf. {\rm \eqref{wilson2}},
to the subspace $\mathcal {CYM}_{\prin,Q}$ 
of  $\mathcal B_{\prin,Q}$ lie in $\Hom_{X_\prin}(\Gamma,\GG)_\prin$, and
the resulting map
\begin{equation}
\wilson^{\sharp}\colon 
\mathcal {CYM}_{\prin,Q}
\longrightarrow
\Hom_{X_\prin}(\Gamma,\GG)_\prin,
\label{singleout2}
\end{equation}
defined on 
 $\mathcal {CYM}_{\prin,Q}$ endowed with its Fr\'echet topology,
is a $\GG$-equivariant homeomorphism.
Furthermore, abstractly, the space
$\mathcal {CYM}_{\prin,Q}$
depends only on the bundle $\prin$
but not on the choices
made to carry out its construction.
Finally, with respect to the $G$-orbit stratifications
on both sides,
$\wilson^{\sharp}$ is an isomorphism of stratified spaces
which, on each stratum,
 restricts to a diffeomorphism onto the corresponding stratum
of the target space.
\end{thm}

We split the proof of \cref{singleout} into Lemmata \ref{lem1} and \ref{lem2}.
In \cref{dep} below, we  establish the \lq\lq Furthermore\rq\rq\ statement 
claiming the independence of the choices.

\begin{lem}
\label{lem1}
The map
 $\wilson^{\sharp}\colon \mathcal {CYM}_{\prin,Q}
\to 
\Hom_{X_\prin}(\Gamma,\GG)_\prin$
given as 
{\eqref{singleout2}} above
is a bijection.
More precisely: For every member $\cchi$ of
$\Hom_{X_\prin}(\Gamma,\GG)_\prin$,
there is a connection $\conn$ on $\prin$
representing a point of
$\mathcal {CYM}_{\prin,Q}$, unique up to based gauge transformations,
such  that the 
composite 
$\Gamma \to \GR \stackrel{\cchi_\conn}\longrightarrow \GG$,
with the injection $\Gamma \to \GR$,
of the
constituent
$\cchi_\conn\colon \GR \to \GG$ in 
the associated pre-reduction map 
{\rm \eqref{redu1}}
 of $\conn$
 relative to   
$Q_{{\wMS}}$ and $\QQQ$ 
in Lemma {\rm \ref{crucial} (iv)}
coincides with $\cchi\colon \Gamma \to \GG$.
\end{lem}

\begin{proof}
Given $\cchi\in \Hom_{X_\prin}(\Gamma,\GG)$,
via the extended homomorphism $\cchi \colon \GR \to \GG$
given by
$\cchi(t[r])=\exp(tX_{\prin})$, for $t \in \mathbb R$,
 the
$\GR$-connection 
form $\omega_{\wMS}\colon \mathrm T \wMS \to \R$ 
on the principal reference
$\GR$-bundle $\prin_{{\wMS}}\colon {\wMS} \to \Sigma$
and the Maurer-Cartan form of $\GG$
yield a $\GG$-connection form
$\omega_{\cchi}\colon \mathrm T \PP_{\cchi} \to \gg$
on the associated principal $\GG$-bundle 
$\prin_{\cchi}\colon 
\PP_{\cchi}\defeq{\wMS} \times_{\cchi} \GG \to \Sigma$.
Below we sometimes write the resulting smooth
principal $\GG$-connection
on $\prin_{\cchi}$ as $A_{\cchi}$.
By construction, the derivative $d\cchi\colon \mathbb R \to \gg$
of 
$\cchi\colon \GR \to \GG$
is given by
$d\cchi(2 \pi)= X_{\prin} \in \gg$. 
Since $\mathrm{curv}_{\omega_{{\wMS}}}=-2\pi \vol_{\Sigma}
\in \mathcal A^2(\Sigma,\mathbb R)$, we conclude that
$\mathrm{curv}_{\omega_\cchi}= -X_{\prin}\cdot\vol_{\Sigma}
\in \mathcal A^2(\Sigma,\zzz)$.

In view of the definition of the connected component
$\Hom_{X_\prin}(\Gamma,\GG)_\prin$ of $\Hom_{X_\prin}(\Gamma,\GG)$,
cf. \cref{top1} above,
the bundles
$\prin$ and $\prin_{\cchi}$ have the same topological type,
that is, are isomorphic.
Since the principal $\GG$-bundles
$\prin$ and $\prin_{\cchi}$ are isomorphic,
there exists a (vertical)
diffeomorphism $\Phi\colon \PP \to \PP_{\cchi}$ that
induces an isomorphism $\prin \to \prin_{\cchi}$
of principal $\GG$-bundles.
Let $\QQQ_{\cchi}=[Q_{{\wMS}}, e]\in \PP_{\cchi}$.
A suitable gauge transformation
of $\prin_{\cchi}$ carries
$\Phi(\QQQ)$ to $\QQQ_{\cchi}$.
Hence we can arrange for the diffeomorphism $\Phi$ to be based
in the sense that
$\Phi(\QQQ)= \QQQ_{\cchi}$.

Pulling back  the connection from \( \omega_\chi \) on \( \prin_\chi \)
along \( \Phi \)  
 yields a  $\GG$-connection $\omega_{\cchi,\Phi}$ 
on $\prin$.
The curvature of \( \omega_{\cchi,\Phi} \) still equals \( -X_{\prin}\cdot\vol_{\Sigma} \).
Another choice $\widehat \Phi\colon P \to P_{\cchi}$ of  
based diffeomorphism
inducing an isomorphism $\prin \to \prin_{\cchi}$
of principal $\GG$-bundles
yields a
 $\GG$-connection form 
$\omega_{\cchi,\widehat \Phi}$
on $\prin$
in such a way that
$\omega_{\cchi,\Phi}$
and
$\omega_{\cchi,\widehat \Phi}$
are based gauge equivalent.
Lemma \ref{crucial} now implies the claim.
\end{proof}

\begin{rema}\label{novel}
The proof of \cref{lem1} essentially involves
the argument for~\cite[Proposition~6.16]{atibottw},
with the connection $\omega_{\wMS}$ on $\prin_{\wMS}\colon \wMS \to
\Sigma$
substituted for the harmonic Yang-Mills connection in~\cite[\S~6]{atibottw}.
As noted in the introduction,
the novelty of our approach resides in the characterization of
the space 
 $\mathcal {CYM}_{\prin,Q}$
merely in terms of  constraints imposed
on the Fr\'echet manifold of smooth based gauge equivalence classes
of all smooth $\GG$-connections on $\prin$.
\end{rema}

\cref{lem1} entails in particular that
the subspace $\mathcal {CYM}_{\prin,Q}$
of $\mathcal B_{\prin,Q}$ cut out by the constraints
\eqref{eq::constraints}
is non-empty.

\begin{rema}
\label{ubc}
By \cref{adapted}, 
the Wilson loop mapping \eqref{wilson1} is smooth, and hence
the Wilson loop mapping \eqref{wilson2} is smooth.
Consequently 
$\wilson^{\sharp}\colon \mathcal {CYM}_{\prin,Q}
\to 
\Hom_{X_\prin}(\Gamma,\GG)_\prin$, cf. \eqref{singleout2}, 
is a continuous bijection
onto the compact subspace $\Hom_{X_\prin}(\Gamma,\GG)_\prin$ 
of the smooth compact manifold
$\Hom(F,\GG)$.
By the strong Uhlenbeck compactness theorem~\cite{MR648356},
for any sequence 
$\{A_j\}$ of smooth Yang-Mills connections with 
constant curvature
(in fact, the hypothesis that the curvature
be uniformly $L^2$-bounded
suffices), there exists a subsequence 
$\{A_i\}$ of smooth Yang-Mills connections and a sequence 
$\{\beta_i\}$
of smooth gauge
transformations such that 
the sequence
$\{\beta_i \cdot A_i\}$ converges 
uniformly
in the $C^{\infty}$-topology 
with all derivatives, that is,
in the Fr\'echet topology,
to a smooth
connection $\conn$.
Hence,
in the Fr\'echet topology,
the space $\mathcal {CYM}_{\prin,Q}$
is as well compact.
Consequently the map \( \wilson^\sharp \) is
a homeomorphism
as asserted.
However, the Fr\'echet slice analysis in
the proof of \cref{lem2} below
avoids that theorem. 
On the other hand, that slice analysis recovers
the statement of the Uhlenbeck compactness theorem in the case at hand.
\end{rema}

\begin{lem}
\label{lem2}
The map $\wilson^{\sharp}\colon \mathcal {CYM}_{\prin,Q}
\to 
\Hom_{X_\prin}(\Gamma,\GG)_\prin$
given as 
{\eqref{singleout2}} above is a $G$-equivariant
homeomorphism that is compatible with the $G$-orbit stratifications
on both sides.
\end{lem}

\begin{proof}
First we show that the bijection \( \rho^\sharp \) is a local homeomorphism.

Consider a smooth connection $\conn$ on $\prin$
that represents a point of 
$\mathcal {CYM}_{\prin,Q}$.
Since the connection $\conn$ is central, the operator $\dif_\conn$ 
of covariant 
derivative turns $\form^*(\Sigma,\adx)$
into a cochain complex. 
We denote the resulting cohomology by $\Ho^*_\conn$ and we 
use the notation
\begin{align}
Z^0_{\conn}&= Z^0_{\conn}(\Sigma,\adx)=\mathrm{ker}(\dif_{\conn}\colon 
\form^0(\Sigma,\adx)
\to
\form^1(\Sigma,\adx))
\\
Z^1_{\conn}&= Z^1_{\conn}(\Sigma,\adx)=\mathrm{ker}(\dif_{\conn}\colon 
\form^1(\Sigma,\adx)
\to
\form^2(\Sigma,\adx))
\\
B^1_{\conn}&= B^1_{\conn}(\Sigma,\adx)=\mathrm{im}(\dif_{\conn}\colon 
\form^0(\Sigma,\adx)
\to
\form^1(\Sigma,\adx))
\\
B^2_{\conn}&= B^2_{\conn}(\Sigma,\adx)=\mathrm{im}(\dif_{\conn}\colon 
\form^1(\Sigma,\adx)
\to
\form^2(\Sigma,\adx)).
\end{align}
In particular, the cokernel of $\dif_\conn\colon 
\form^1(\Sigma,\adx) \to \form^2(\Sigma,\adx)$
yields the associated second cohomology group $\Ho^2_\conn$.
We choose a complement
$\mathcal H^2_\conn\subseteq  \form^2(\Sigma,\adx)$
for $B^2_\conn\subseteq  \form^2(\Sigma,\adx)$
such that the direct sum decomposition
$\form^2(\Sigma,\adx) = B^2_\conn \oplus \mathcal H^2_\conn$
is a homeomorphism in the Fr\'echet topology.
Then the projection 
$\form^2(\Sigma,\adx) \to \Ho^2_\conn$,
restricted to
$\mathcal H^2_\conn$, is an isomorphism of finite-dimensional vector spaces.

Let
$S_\conn \subseteq \Conn$ 
be a Fr\'echet slice
for the action of the group 
$\Gau$
of gauge transformations on $\Conn$. Then
\begin{equation}
\mathrm T_\conn\Conn =B^1_\conn \oplus \mathrm T_\conn(S_\conn).
\end{equation}
Consider the smooth map
\begin{equation}
h\colon S_\conn \to B^2_\conn,\ \conn' \mapsto \curv_{\conn'} 
\mod \mathcal H^2_\conn,\ 
\conn' \in S_\conn.
\label{smooth1}
\end{equation}
The derivative \( \tangent_{\conn'} h\colon \mathrm T_{\conn'}S_\conn \to B^2_\conn \)
at the point $\conn'$ of $S_\conn$ is the composite
\begin{equation}
\mathrm T_{\conn'}S_\conn \stackrel{\dif_{\conn'}|}\longrightarrow 
\form^2 
\stackrel{\mathrm{proj}}
\longrightarrow
B^2_\conn.
\end{equation}
By construction, the derivative $\tangent_\conn h = \dif_\conn$ at the point $\conn$ is surjective.
Our aim is to conclude that, 
for an open neighborhood $U$ of $\conn$ in $S_\conn$,
the pre-image
$h^{-1}(0)\cap U$
is a smooth manifold, necessarily finite-dimensional,
 having as tangent space at $\conn$
the finite-dimensional vector  space 
$Z^1_\conn\cap \mathrm T_\conn S_\conn\cong \Ho^1_\conn$.
To this end,
we use the Nash--Moser theorem to show that \( h \) is a submersion 
at \( \conn \).
Hence we need to verify that the derivative \( \tangent_ {\conn'} h \) is 
surjective in an open neighborhood of \( \conn \) in \( S_{\conn} \) and 
that the family of right inverses is tame smooth.

Let $\conn + \alpha$ be a point of $\Conn$, where
 $\alpha \in \form^1(\Sigma,\adx)$,
and suppose that $\conn + \alpha \in S_\conn$.
Consider the  operators
\begin{align*}
	L_\alpha&\colon \form^1(\Sigma,\adx) \longrightarrow \form^2(\Sigma,\adx), 
\quad \beta \mapsto \dif_{\conn + \alpha} \beta = \dif_\conn \beta + [\alpha \wedge \beta],
\\
T_\alpha &= L_\alpha \circ L_0^* = \dif_{\conn + \alpha} \dif^*_\conn\colon \form^2(\Sigma,\adx) \longrightarrow \form^2(\Sigma,\adx).
\end{align*}
The operator  
\( T_\alpha  \) 
satisfies the identity 
\[
 T_\alpha (\eta) = \triangle_\conn(\eta) + [\alpha \wedge \dif^*_\conn (\eta)]\quad 
(\eta \in \form^2(\Sigma,\adx)) .
\]
In particular, \( T_0 \) is the covariant Laplacian on $\form^2(\Sigma,\adx)$ 
relative to $\conn$
and so 
is an elliptic operator of index zero.
Since the symbol map is continuous, 
for small \( \alpha \),
the deformed operator \( T_\alpha \) 
is still elliptic and has index zero.
We claim that the induced operator
\begin{equation}
B^2_\conn \stackrel{\mathrm {inj}}\longrightarrow
\form^2(\Sigma,\adx)
\stackrel{ T_\alpha }
\longrightarrow
\form^2(\Sigma,\adx)
\stackrel{\mathrm{proj}} 
\longrightarrow B^2_\conn
\label{indop}
\end{equation}
is invertible.
In view of the Fredholm alternative,
since   \( T_\alpha \) has index zero,
 it is enough to show that
\( \ker T_\alpha = \Ho^2_\conn \).
The identity 
\( {(\dif_\conn \dif^*_\conn \eta, \eta)}_{L^2} 
= \lVert \dif^*_\conn \eta \rVert^2_{L^2} \) 
shows that \( \Ho^2_\conn = \ker T_0 = \ker \dif_\conn  \dif^*_\conn \) 
coincides with the kernel of \( \dif^*_\conn \).
Hence \( \ker T_0 \subseteq \ker T_\alpha \). 
The converse inclusion follows from the upper semi-continuity of the 
dimension of the kernel of an elliptic 
operator~\cite[Corollary~19.1.6]{MR2304165}, 
that is, \( \dim \ker T_\alpha \leq \dim \ker T_0 \) 
for small \( \alpha \).
Hence the induced operator \eqref{indop}
is invertible and, with a slight abuse of notation, we denote
this operator on $B^2_\conn$  by \( T_\alpha \) as well.
The inverse operator \( T^{-1}_\alpha \) 
arises from the Green's operator and, 
by~\cite[II.3.3.3~Theorem~p.~158]{MR656198}, 
is therefore tame.

Since \( L_\alpha \circ (L_0^* \circ T^{-1}_\alpha) = \id_{B^2_\conn} \)
for small $\alpha$, we obtain the tame family
\( \{L_0^* \circ T^{-1}_\alpha\}_\alpha \) of right inverses for the family 
\( \{\tangent_{\conn + \alpha} h\}_\alpha \).
Hence the Nash--Moser inverse function theorem implies that \( h \) is a 
submersion near \( \conn \).
In other words, there exists an open neighborhood $U$ of $\conn$ in $S_\conn$ 
such that the pre-image $h^{-1}(0) \cap U$ is a smooth manifold $M_\conn$ 
having as tangent space at $\conn$ the finite-dimensional vector space 
$Z^1_\conn\cap \mathrm T_\conn S_\conn\cong \Ho^1_\conn$.
Since the canonical map \( \mathcal H^2_\conn \to \Ho^2_\conn\) is 
a linear isomorphism,
the manifold \( M_\conn \) contains connections \( \conn' \in S_\conn \) 
that are \( C^\infty \)-close to \( \conn \) and have curvature map
\( \ast\, \curv_{\conn'}: \PP \to \gg \) constant on the holonomy 
bundle of \( \conn \).

The action of the group $\Gau$ of gauge transformations on $\form_{\prin}$
restricts to a smooth action of the stabilizer subgroup $\Stab_{\conn}$ 
of $\conn$ on $M_{\conn}$. Moreover, in  a $\Stab_{\conn}$-neighborhood 
of $\conn$ in  $M_{\conn}$, the projection to $\mathcal B_{\prin,Q}$ 
is a diffeomorphism onto a smooth finite-dimensional 
$\Stab_{\conn}$-submanifold  $M_{\conn,Q}$ of $\mathcal B_{\prin,Q}$.
The restriction  to $M_{\conn,Q}$ of the Wilson loop mapping 
\( \wilson_\flat\colon \mathcal B_{\prin,Q} \to \GG^{2\ell} \)
given as \eqref{wilson2}  above yields a smooth map
\begin{equation}
\wilson_{\natural}\colon M_{\conn,Q} \to 
\Hom(F,\GG)\cong \GG^{2\ell}
\label{smooth11}
\end{equation}
having injective tangent map 
\[
\tangent_{[\conn]} \wilson_{\natural}\colon 
\mathrm T_{[\conn]}M_{\conn,Q} \to 
\mathrm T_{\wilson(\conn)}\Hom(F,\GG)
\]
at the point $[\conn]$ and, up to left translation by
$\wilson(\conn) \in \Hom(F,\GG) \cong \GG^{2\ell}$, 
that tangent map takes the form
\begin{equation}
\tangent_{[\conn]}\wilson_{\natural}\colon 
\mathrm T_{[\conn]}M_{\conn,Q} \to 
\gg^{2\ell}.
\label{form1}
\end{equation}
We shall shortly justify the injectivity claim.
Passing, if need be, to a smaller open submanifold of
$M_{\conn,Q}$ containing the point $[\conn]$, 
we can arrange for the image 
\( M_{\wilson(\conn),Q} \defeq \wilson_\natural (M_{\conn, Q}) 
\subseteq \Hom(F,\GG)
\)
to be an ordinary (embedded) submanifold of \( \Hom(F, \GG) \). 

Since the connection $\conn$ is central,
the homomorphism $\cchi_{\conn}=\wilson(\conn)\colon F \to \GG$ 
descends to a homomorphism
$\pio \to \GG/\ZZZ\cong \overline S$ and hence induces,
via the adjoint action of $\GG$,
a $\pio$-module structure on $\gg$. 
We write the resulting
$\pio$-module as $\gg_{\conn}$ or $\gg_{\cchi_{\conn}}$.
Then the cellular cochain complex of $\Sigma$ with local coefficients
is defined and takes the form
\begin{equation}
\begin{CD}
\gg_{\cchi_{\conn}}
@>{d_{\cchi_{\conn}}}>>
\gg_{\cchi_{\conn}}^{2\ell}
@>{d_{\cchi_{\conn}}}>>
\gg_{\cchi_{\conn}}.
\end{CD}
\label{coch2}
\end{equation}
Twisted integration yields a morphism
of cochain complexes
\begin{equation}
\begin{CD}
\form^0(\Sigma,\adx)
@>{\dif_{\conn}}>>
\form^1(\Sigma,\adx)
@>{\dif_{\conn}}>>
\form^2(\Sigma,\adx)
\\
@VVV
@V{\int}VV
@VVV
\\
\gg_{\cchi_{\conn}}
@>{d_{\cchi_{\conn}}}>>
\gg_{\cchi_{\conn}}^{2\ell}
@>{d_{\cchi_{\conn}}}>>
\gg_{\cchi_{\conn}} 
\end{CD}
\end{equation}
inducing an isomorphism on cohomology, in particular an isomorphism
$\Ho^1_\conn(\Sigma, \adx) \to \Ho^1(\Sigma,\gg_{\cchi_{\conn}})$.
Given $\chi \colon \Gamma \to \GG$ such that the value $\chi([r])$ lies in 
$\ZZZ$ (the connected component of the center of $\GG$), we  use the notation
\begin{align}
Z^0_{\cchi}&= Z^0_{\cchi}(\Sigma,\ggcchi)=\mathrm{ker}(d_{\cchi}\colon 
C^0(\Sigma,\ggcchi)
\to
C^1(\Sigma,\ggcchi))
\\
Z^1_{\cchi}&= Z^1_{\cchi}(\Sigma,\ggcchi)=\mathrm{ker}(d_{\cchi}\colon 
C^1(\Sigma,\ggcchi)
\to
C^2(\Sigma,\ggcchi))
\\
B^1_{\cchi}&= B^1_{\cchi}(\Sigma,\ggcchi)=\mathrm{im}(d_{\cchi}\colon 
C^0(\Sigma,\ggcchi)
\to
C^1(\Sigma,\ggcchi))
\\
B^2_{\cchi}&= B^2_{\cchi}(\Sigma,\ggcchi)=\mathrm{im}(d_{\cchi}\colon 
C^1(\Sigma,\ggcchi)
\to
C^2(\Sigma,\ggcchi)).
\end{align}
By \cref{adapted}, up to left translation
by the value $\wilson(\conn) \in \Hom(F,\GG)\cong \GG^{2\ell}$,
the twisted integration mapping
$\int\colon \form^1(\Sigma,\adx) \to \gg_{\cchi_{\conn}}^{2\ell}
$
is the derivative~\eqref{diff2}
of the  Wilson loop mapping \( \wilson \)
at the point $\conn$ of $\form_{\prin}$.
Hence the image of \( \mathcal H^1_{\conn} = \mathrm T_{[\conn]}M_{\conn,Q} \)
under 
$\tangent_{[\conn]}  
\wilson_{\natural}\colon T_{[\conn]}M_{\conn,Q} \to \gg^{2\ell}$
is a linear subspace of $\gg_{\cchi_{\conn}}^{2\ell}$
which lies in the subspace
$Z^1_{\cchi_{\conn}} \subseteq \gg_{\cchi_{\conn}}^{2\ell}$
of 1-cocycles relative to 
$d_{\cchi_{\conn}}\colon  \gg_{\cchi_{\conn}}^{2\ell} \to  \gg_{\cchi_{\conn}}$
and maps, under the projection from
$Z^1_{\cchi_{\conn}}$ to the cohomology group 
$\Ho^1(\Sigma,\gg_{\cchi_{\conn}})$,
 isomorphically onto
that cohomology group.
Consequently, \( \tangent_{[\conn]} \wilson_{\natural} \) is injective and thus
\( \wilson_\natural \)
is a diffeomorphism
$M_{\conn,Q} \to M_{\wilson(\conn),Q}$
onto its image $M_{\wilson(\conn),Q}$,
at least in a neighborhood of
$[\conn]$ in $M_{\conn,Q}$.

The evaluation map \( \mathrm{ev}_{\QQQ}\colon \Gau \to \GG \)
identifies
the stabilizer $\Stab_{\conn}$ of $\conn$ 
 with the stabilizer
$\Stab_{\cchi_{\conn}}\subseteq \GG$ of the value
$\cchi_{\conn} =\wilson(\conn) \in M_{\wilson(\conn),Q}$.
Thus the diffeomorphism 
\( \wilson_\natural\colon M_{\conn,Q} \to M_{\wilson(\conn),Q} \)
extends to a diffeomorphism
\begin{equation}
\wilson_{\natural}\colon \GG \times_{\Stab_{\conn}}M_{\conn,Q} 
\to \GG \times_{\Stab_{\cchi_{\conn}}} M_{\wilson(\conn),Q}
\subseteq
\Hom(F,\GG)\cong \GG^{2\ell}.
\label{smooth2}
\end{equation}
Consequently the restriction
\begin{equation}
\big(\GG \times_{\Stab_{\conn}}M_{\conn,Q}\big) 
\cap \mathcal{CYM}_{\prin,Q}
\to
\big( \GG \times_{\Stab_{\cchi_{\conn}}} M_{\wilson(\conn),Q}\big)
\cap
\Hom_{X_\prin}(\Gamma,\GG)_\prin
\end{equation}
is a homeomorphism.
By construction, this homeomorphism is the restriction,
to the intersection
$
\left(\GG \times_{\Stab_{\conn}}M_{\conn,Q}\right) 
\cap \mathcal{CYM}_{\prin,Q}$,
of the map 
\( \wilson^{\sharp} \) given as
\eqref{singleout2} above.

Since our reasoning works for any point $[\conn]$ of
$\mathcal {CYM}_{\prin,Q}$, we conclude that 
\( \wilson^{\sharp} \)
is a homeomorphism.
Moreover, the above Fr\'echet slice analysis shows that \( \wilson^{\sharp} \) is compatible with the orbit type stratifications
on both sides and 
that,
 on each stratum,
 \( \wilson^{\sharp} \) restricts to a diffeomorphism onto the corresponding stratum
of the target space.
\end{proof}

\begin{rema}
\label{rem1}
Let $\GG = \mathrm U(1)$ and 
endow our
reference 
$\mathrm U(1)$-bundle $\prin_{\MS} \colon {\MS} \to \Sigma$
introduced in \cref{top1} above with 
a connection $\conn$
whose curvature form $\curv_\conn$ equals $-2 \pi i \vol_\Sigma$, that is,
represents a point of
$\mathcal{CYM}_{\prin_{\MS},Q}$.
Under these circumstances, the pre-reduction map of $\conn$
relative to $Q_{\wMS}$ and $Q_{\MS}$
induces a gauge transformation of $\prin_{\MS}$ that identifies
the reference connection 
chosen in \cref{top1}
with $\conn$. In particular, the space
$\mathcal{CYM}_{\prin_{\MS},Q}$ reduces to a point.
This recovers the fact that a complex line bundle on $\Sigma$
has a  harmonic connection that is unique up to gauge transformations.
\end{rema}

\section[Reconstruction of the space
of gauge equivalence classes of all Yang-Mills connections]{Fr\'echet-reconstruction of the space
of gauge equivalence classes of all Yang-Mills connections}
\label{all}
In this section we reconstruct the space of gauge equivalence classes of 
Yang-Mills connections on $\Sigma$ within our Fr\'echet setting.
Our procedure parallels the corresponding classical description of the
topological decomposition of the space of Yang-Mills connections
given in \cite[Section 6]{atibottw}.

The obvious surjection $\mathrm{can}\colon F \to \Gamma$
defined on the free group $F$ on the chosen generators
$x_1,y_1,\dots, x_\ell,y_\ell$ of the fundamental group $\pi_1(\Sigma,Q)$ 
extends to the surjective homomorphism
\begin{equation}
\mathbb R \times F \longrightarrow \GR,\ 
(t,w) \longmapsto t[r]\,\mathrm{can}(w),\ t \in \mathbb R,\ w \in F.
\end{equation}
This surjective homomorphism, in turn, induces an injection
of $\Hom(\GR, \GG)$ into $\Hom(\mathbb R \times F, \GG)$.
The restriction to the central copy of $\mathbb R$ 
in $\GR$ generated by $[r]$
of 
a homomorphism $\varphi\colon \mathbb R \times F \to \GG$
is a $1$-parameter subgroup of $\GG$ and therefore of the kind
$t \mapsto \mathrm{exp}(t X_\varphi)$ ($t \in \mathbb R$)
for some  uniquely determined member $X_{\varphi}$ of $\gg$.
Hence
the assignment to $\varphi\in\Hom(\mathbb R \times F, \GG)$
of
$
(X_{\varphi},\varphi(x_1),\varphi(y_1),\dots, \varphi(x_\ell),\varphi(y_\ell)) 
$
injects $\Hom(\mathbb R \times F, \GG)$ into $\gg \times \GG^{2\ell}$ 
and hence yields an injection
\begin{equation}
\Hom(\GR, \GG) \longrightarrow \gg \times \GG^{2\ell}.
\label{inj1}
\end{equation}
We endow $\Hom(\GR, \GG)$ with the induced topology.

Let $X\in \gg$. We denote by $\GG_X\subseteq \GG$ the stabilizer of $X$
under the adjoint action, necessarily
a closed connected subgroup of $\GG$,
cf.~\cite[\S 32 Theorem 32.16 p.~259]{MR770935},
and we denote by $\Hom_X(\GR,\GG)$
 the space, if non-empty, of homomorphisms
$\cchi$  from $\GR$ to $\GG$ 
having the property that $\cchi(t[r])=\exp(tX)$.
Since the copy of $\mathbb R$ in $\GR$ 
generated by $[r]$ is central, the values of any $\varphi \in 
\Hom_X(\GR,\GG)$ lie in
$\GG_X$, that is, the canonical injection map from
$\Hom_X(\GR,\GG_X)$
to $\Hom_X(\GR,\GG)$
is a homeomorphism.
Given an adjoint orbit $\mathcal O \subseteq \gg$,
we denote by $\Hom_{\mathcal O}(\GR,\GG)$
the space of homomorphisms
$\cchi$ from $\GR$ to $\GG$ 
having the property that  $\cchi(t[r])=\exp(tY)$, 
for some $Y\in \mathcal O$.
The space $\Hom_{\mathcal O}(\GR,\GG)$ is
the total space of a fiber bundle over $\mathcal O$
having compact fiber,
the fiber over
$Y \in \mathcal O$ being the compact space
$\Hom_Y(\GR,\GG_Y)$,
and so $\Hom_{\mathcal O}(\GR,\GG)$ is therefore itself compact.

We return to our principal $\GG$-bundle $\prin\colon \PP \to \Sigma$.
Given an adjoint orbit $\mathcal O$, we denote by
$\Hom_{\mathcal O}(\GR,\GG)_\prin$ the subspace, if non-empty, of 
$\Hom_{\mathcal O}(\GR,\GG)$ 
that consists of the homomorphisms $\varphi\colon \GR \to \GG$
in $\Hom_{\mathcal O}(\GR,\GG)$ which have the property that the 
associated principal $\GG$-bundle
$\prin_{\widetilde M} \times_{\varphi}\GG$
is topologically equivalent to $\prin$.
When $\mathcal O$ consists of a single point $X$, we
use the notation 
$\Hom_{X}(\GR,\GG)_\prin$.
We then denote by
$\Hom(\GR,\GG)_\prin$
 the disjoint union of the spaces of the kind 
$\Hom_{\mathcal O}(\GR,\GG)_\prin$,
where $\mathcal O$ ranges over the adjoint orbits such that
$\Hom_{\mathcal O}(\GR,\GG)_\prin$
is non-empty.
In terms of the notation thus established,
the restriction mapping 
\begin{equation}
\Hom_{X_{\prin}}(\GR,\GG)
\longrightarrow 
\Hom_{X_{\prin}}(\Gamma,\GG)
\label{rest1}
\end{equation} 
yields a homeomorphism from 
\( \Hom_{X_{\prin}}(\GR,\GG)_\prin \)
onto the space
$\Hom_{X_\prin}(\Gamma,\GG)_\prin$ used in 
\cref{singleout} 
above.

Recall 
from the proof of \cite[Theorem 6.7]{atibottw} that Yang--Mills connections 
on \( \prin \) arise from central Yang--Mills connections on certain 
subbundles. Indeed, consider a connection \( \conn \) on \( \prin \) that 
satisfies the Yang--Mills equation. 
Then the \( \GG \)-equivariant map 
\( \ast\, \curv_\conn: \PP \to \gg \) is constant along horizontal paths. 
Let $X= \ast\, \curv_\conn (Q_\PP)$ and
\( \PP_{\conn, X} = \{ p \in \PP: \ast\, \curv_\conn (p) = X \} \).
The bundle projection $\prin$ restricts to a
principal \( \GG_X \)-bundle projection 
\( \prin_X\colon \PP_{\conn, X} \to \Sigma\), and the connection $\conn$
restricts to a connection $\conn_X$ on $\prin_X$ having the property that
$\curv_{A_X}= -X \cdot\vol_\Sigma$.

For \( X \in \gg \), denote by \( \mathcal {YM}_{\prin,Q,X} \) 
the subspace, if non-empty, 
of \( \mathcal B_{\prin,Q} \) that consists of the points \( [A] \) 
characterized as follows: There exists a reduction of structure group  
$\prin_X\colon \PP_X \to \Sigma$
of $\prin$ to $\GG_X$ 
that is based in the sense that,
relative to the associated injection
of $\PP_X$ into the total space $\PP$ of $\prin$,
the base point $\QQQ$ of $\PP$ lies in $\PP_X$, such that some connection
$\conn$ representing the class $[A]$
restricts to a connection $A_X$ on $\prin_X$
 having the property that
$\curv_{A_X}= -X \cdot\vol_\Sigma$.
The residual \( \GG \)-action then corresponds to the adjoint action 
on \( \gg \), i.e., \( x \in \GG \) maps \( \mathcal {YM}_{\prin,Q,X} \) 
to \( \mathcal {YM}_{\prin,Q,\AdAction_x X} \). 
This leads us to define, for a given adjoint orbit 
$\mathcal O$, the \( \GG \)-invariant subspace 
\( \mathcal {YM}_{\prin,Q,\mathcal O}\) of \( \mathcal B_{\prin,Q} \) 
to be the union of the spaces 
\( \mathcal {YM}_{\prin,Q,X} \) for \( X \in \mathcal{O} \).
The above argument  shows that the space 
$\mathcal {YM}_{\prin,Q}\subseteq \mathcal B_{\prin,Q}$ of based gauge 
equivalence 
classes of Yang--Mills connections on \( \prin \) is the disjoint union of 
the spaces of the kind 
$\mathcal {YM}_{\prin,Q,\mathcal O}$,
where $\mathcal O$ ranges over those adjoint orbits such that
$\mathcal {YM}_{\prin,Q,\mathcal O}$ is non-empty.
The space 
$\mathcal {YM}_{\prin,Q}$ is our \emph{Fr\'echet version of the space
of based gauge equivalence classes of Yang-Mills connections on
$\prin$}.
Likewise we take the Fr\'echet version $\mathcal {YM}_{\GG,\Sigma,Q}$ 
of the space
of based gauge equivalence classes of all Yang-Mills connections over
$\Sigma$ relative to $\GG$
to be the disjoint union of the spaces of the kind 
$\mathcal {YM}_{\prin,Q}$,
where $\prin$ ranges over all topological types of $\GG$-bundles
on $\Sigma$.
By construction,
each connected component 
of $\mathcal {YM}_{\prin,Q}$
is characterized by Fr\'echet constraints, and the same is true of
$\mathcal {YM}_{\GG,\Sigma,Q}$.
In this sense, the spaces
$\mathcal {YM}_{\prin,Q}$
and
$\mathcal {YM}_{\GG,\Sigma,Q}$
are entirely characterized by Fr\'echet constraints.
We say that a space of the kind
$\mathcal {YM}_{\prin,Q,\mathcal O}$
and, likewise, one the kind
$\Hom_{\mathcal O}(\GR,\GG)_\prin$
 {\em is 
defined\/} when its defining condition is non-vacuously
satisfied.

\begin{thm}
\label{count}
Given an adjoint orbit $\mathcal O$ in $\gg$, the space
$\mathcal {YM}_{\prin,Q,\mathcal O}$ is 
defined
 if and only if
the space
$\Hom_{\mathcal O}(\GR,\GG)_\prin$ is defined
and, when this happens to be the case, 
the Wilson loop mapping
$\wilsonn \colon \mathcal B_{\prin,Q} \to \GG^{2\ell}$,
cf. \eqref{wilson2},
induces a $\GG$-equivariant homeomorphism 
\begin{equation}
\mathcal {YM}_{\prin,Q,\mathcal O}
\longrightarrow
\Hom_{\mathcal O}(\GR,\GG)_\prin.
\label{orbit}
\end{equation}
Hence
the Wilson loop mapping~$\wilsonn$
induces a $\GG$-equivariant homeomorphism 
\begin{equation}
\mathcal {YM}_{\prin,Q}
\longrightarrow
\Hom(\GR,\GG)_\prin.
\label{Sigma}
\end{equation}
Furthermore,
 as $\prin$
ranges over all topological types of principal $\GG$-bundles
on $\Sigma$, 
the Wilson loop mapping~$\wilsonn$  
induces a homeomorphism 
\begin{equation}
\mathcal {YM}_{\GG,\Sigma,Q}
\longrightarrow
\Hom(\GR,\GG).
\label{all1}
\end{equation}
Moreover, abstractly the spaces $\mathcal {CYM}_{\prin,Q}$ and
$\mathcal {CYM}_{\GG,\Sigma,Q}$ do not depend on the choices
made to carry out their construction.
Finally, each of the above homeomorphisms is compatible with the
$G$-orbit type stratifications and is, hence,
an isomorphism of stratified spaces.
\end{thm}

\begin{proof}
Consider $X \in \gg$  such that $\Hom_X(\GR,\GG)$ is non-empty and choose
$\cchiz\in \Hom_X(\GR,\GG)$. We  denote by $\Hom_{X}(\GR,\GG)_\cchiz$
the connected component of $\Hom_X(\GR,\GG)$ that contains $\cchiz$.
Since  \( [r] \) is central in \( \GR \), the values of any 
$\varphi \in \Hom_{X}(\GR,\GG)_\cchiz$ lie in $\GG_X$ whence
the canonical injection from $\Hom_{X}(\GR,\GG_X)_\cchiz$
to $\Hom_{X}(\GR,\GG)_\cchiz$ is a homeomorphism. The associated bundle
\begin{equation}
\prin_{\cchiz}= \prin_{\widetilde M} \times_{\cchiz}\GG_X 
\colon \PP_{\cchiz}= \widetilde M \times_{\cchiz}\GG_X \longrightarrow \Sigma
\label{asso2}
\end{equation}
is a principal $\GG_X$-bundle, and the canonical restriction map
\begin{equation}
\Hom_{X}(\GR,\GG_X)_\cchiz \longrightarrow
\Hom_{X}(\Gamma,\GG_X)_{\prin_\cchiz}
\end{equation}
is a homeomorphism.
We take the point $Q_{\cchiz} =[(Q_{\wMS},e)]$ 
as the base point of $\PP_{\cchiz}$. By 
\cref{singleout}, 
\begin{equation}
\wilson^\sharp \colon \mathcal{CYM}_{\prin_{\cchiz},Q}
\longrightarrow \Hom_{X}(\Gamma,\GG_X)_{\prin_\cchiz}
\end{equation}
is a \( \GG \)-equivariant homeomorphism.

Extension of the structure group $\GG_X$ to $\GG$ yields the principal 
$\GG$-bundle
\begin{equation}
\prin_{\cchiz, \GG}\colon \PP_{\cchiz}\times_{\GG_X}\GG \to \Sigma
\end{equation}
on $\Sigma$.
The action $\mathcal G_{\prin_{\cchiz,\GG}} \times \form_{\prin_{\cchiz,\GG}} \to \form_{\prin_{\cchiz,\GG}}$
of  the group $\mathcal G_{\prin_{\cchiz,\GG}}$ 
of smooth gauge transformations of
$\prin_{\cchiz,\GG}$
on the space $\form_{\prin_{\cchiz,\GG}}$
of smooth connections 
on $\prin_{\cchiz,\GG}$
induces a $\GG$-equivariant homeomorphism
\begin{align*}
\GG \times_{\GG_X} \mathcal{CYM}_{\prin_{\cchiz},Q}  &
\longrightarrow \mathcal{YM}_{\prin_{\cchiz,\GG},\mathcal O,Q} .
\end{align*}
Likewise, the conjugation action
\begin{align*}
\GG \times \Hom(\Gamma,\GG) \longrightarrow \Hom(\Gamma,\GG),
\ 
(x,\cchi) \mapsto {}^x \cchi, \ {}^x \cchi(y)=x\cchi(y) x^{-1},\ 
x\in \GG, y \in \Gamma,\cchi \in  \Hom(\Gamma,\GG),
\end{align*}
of $\GG$ on  $\Hom(\Gamma,\GG)$ 
induces a $\GG$-equivariant homeomorphism
\begin{align*}
\GG \times_{\GG_X}\Hom_{X}(\Gamma,\GG_X)_{\prin_\cchiz}
&
\longrightarrow \Hom_{\mathcal O}(\Gamma,\GG)_{\prin_{\cchiz, \GG}}.
\end{align*}
Consequently, the Wilson loop mapping induces
a $\GG$-equivariant homeomorphism 
\begin{equation*}
 \mathcal {YM}_{\prin_{\cchiz,\GG},Q,\mathcal O}
\longrightarrow
\Hom_{\mathcal O}(\GR,\GG)_{\prin_{\cchiz,\GG}}
\end{equation*}
of the kind \eqref{orbit}.

The rest of the proof is straightforward. We leave the details to the reader.
\end{proof}

\begin{rema}
\cref{count}
is a based version
of~\cite[Theorem 6.7]{atibottw}, but phrased
in our Fr\'echet setting.
We can paraphrase \cref{count} by saying that
our Fr\'echet construction precisely recovers
the Atiyah--Bott spaces of based gauge equivalence classes
of Yang-Mills connections.
\end{rema}

To spell out what \cref{singleout,count} entail for the 
ordinary moduli spaces of Yang-Mills connections, we use the notation
\begin{align}
\mathcal {CYM}_{\prin}&=\GG\backslash \mathcal {CYM}_{\prin,Q}
\\
\mathcal {YM}_{\prin}&=\GG\backslash\mathcal {YM}_{\prin,Q}
\\
\mathcal {YM}_{\GG,\Sigma}&=\GG\backslash\mathcal {YM}_{\GG,\Sigma,Q}
\\
\intertext{and}
\mathrm{Rep}_{X_{\prin}}(\GR,\GG)_\prin&=
\GG\backslash\Hom_{X_{\prin}}(\GR,\GG)_\prin\,
(\cong \GG\backslash\Hom_{X_\prin}(\Gamma,\GG)_\prin)
\\
\mathrm{Rep}(\GR,\GG)_\prin&=
\GG\backslash\Hom(\GR,\GG)_\prin
\\
\mathrm{Rep}(\GR,\GG)&=
\GG\backslash\Hom(\GR,\GG) .
\end{align}
Plainly, the homeomorphism \( \wilson^\sharp \), cf.  \eqref{singleout2},
induces homeomorphisms
\begin{align}
\mathcal {CYM}_{\prin}
&\longrightarrow
\mathrm{Rep}_{X_{\prin}}(\GR,\GG)_\prin,
\label{singleout31}
\\
\mathcal {YM}_{\prin}
&\longrightarrow
\mathrm{Rep}(\GR,\GG)_\prin,
\label{singleout4}
\\
\mathcal {YM}_{\GG,\Sigma}
&\longrightarrow
\mathrm{Rep}(\GR,\GG),
\label{singleout5}
\end{align}
and these homeomorphism preserve the stratifications on both sides.
In view of \cref{singleout,count}, the spaces
$\mathcal {CYM}_{\prin}$,
$\mathcal {YM}_{\prin}$,  and
$\mathcal {YM}_{\GG,\Sigma}$
do not depend on the choices made to construct the
Wilson loop mapping.
The orbit spaces $\mathcal {CYM}_{\prin}$,
 $\mathcal {YM}_{\prin}$, and
$\mathcal {YM}_{\GG,\Sigma}$ recover,
respectively,
the moduli space of gauge equivalence classes of
central Yang-Mills connections on $\prin$,
that of Yang-Mills connections on $\prin$,
and that of all Yang-Mills 
connections on $\Sigma$ relative to $\GG$.

\section{Dependence on the Choices}
\label{dep}
The Wilson loop mapping \( \wilson_\flat \), cf. \eqref{wilson2},
and the resulting map \( \wilson^\sharp \), cf. \eqref{singleout2},
depend on the choices
of the point $\QQQ$ of the total space $\PP$ of
$\prin$ and the canonical system
$u_1,v_1,\dots,u_\ell,v_\ell$
of curves in $\Sigma$.
Let $\overline u_1,\overline v_1,\dots,\overline u_\ell,\overline v_\ell$
be another canonical system of curves in $\Sigma$ that start at $Q$.
Their based homotopy classes 
$\overline x_1,\overline y_1,\dots,\overline x_\ell,\overline y_\ell$
constitute a system of generators of the fundamental group
$\pio=\pi_1(\Sigma,Q)$ of $\Sigma$ at $Q$.
Define the relator by $\overline r = 
\left[\overline x_1,\overline y_1\right] \cdot
\dots
\cdot
\left[\overline x_\ell,\overline y_\ell\right] 
$. The association
\[
 (u_1, v_1,\dots, u_\ell, v_\ell)
\longrightarrow
(\overline u_1,\overline v_1,\dots,\overline u_\ell,\overline v_\ell)
\]
induces an automorphism of $\Gamma$ and one of $\pio$ that make the diagram
\begin{equation}
\begin{CD}
0
@>>>
\mathbb Z\langle [r]\rangle
@>>>
\Gamma
@>>>
\pio
@>>>
1
\\
@.
@|
@VVV
@VVV
@.
\\
0
@>>>
\mathbb Z\langle [r]\rangle
@>>>
\Gamma
@>>>
\pio
@>>>
1
\end{CD}
\label{3.22}
\end{equation}
commutative.
Since the closed paths
$[u_1, v_1] \cdot \ldots\cdot [u_\ell, v_\ell]$ and
$[\overline u_1, \overline v_1] \cdot \ldots\cdot [\overline u_\ell, \overline v_\ell]$
acquire the same orientation,
the two members $[r]$ and $[\overline r]$ of $\Gamma$ coincide.
This is a consequence of the fact that, up to within isomorphism, 
the group $\Gamma$
is the uniquely determined Schur cover of $\pio=\pi_1(\Sigma,Q)$.
Hence the identity symbol in \eqref{3.22} makes sense.
That automorphism, in turn, determines a member of
the mapping class group of $\Sigma$. These observations yield a proof 
of the following.

\begin{prop}
\label{singleout33}
In the situation of 
{\rm \cref{singleout}} above,
the base point $\QQQ$ of the total space $\PP$ of $\prin$
being fixed, two isomorphisms  
$\mathcal {CYM}_{\prin,Q}
\to
\Hom_{X_\prin}(\Gamma,\GG)_\prin$
of stratified spaces that arise from
two Wilson loop mappings of the kind \eqref{singleout2}
associated to distinct choices of
canonical systems
of curves in $\Sigma$
differ by a homeomorphism
$\Hom_{X_\prin}(\Gamma,\GG)_\prin
\to
\Hom_{X_\prin}(\Gamma,\GG)_\prin$,
necessarily an automorphism of stratified spaces,
induced by an automorphism of $\Gamma$ that represents a
member of the mapping class group of $\Sigma$. \qed 
\end{prop}

\begin{prop}
\label{singleout44}
In the situation of Theorem {\rm \ref{singleout}} above,
the base point $Q \in \Sigma$ being fixed,
two isomorphisms of stratified spaces  of the kind~\eqref{singleout2}
that arise from
the Wilson loop mappings
associated to the same choice of
canonical system
of curves in $\Sigma$
but one relative to  $\QQQ$
and the other one
relative to  $\QQQ \cdot x$
for some $x \in \GG$
differ by the homeomorphism
$\Hom_{X_\prin}(\Gamma,\GG)_\prin
\to
\Hom_{X_\prin}(\Gamma,\GG)_\prin
$
induced by conjugation with $x \in \GG$,
necessarily an automorphism of stratified spaces. \qed 
\end{prop}

\begin{prop}
\label{singleout55}
In the situation of Theorem {\rm \ref{singleout}} above,
with a new choice of
 base points $Q'$ of $\Sigma$ and
$Q'_{\MS}$ of $\MS$, a choice of a homotopy class of paths in
$\MS$ from $Q_{\MS}$ to $Q'_{\MS}$ induces an  isomorphism
of stratified spaces
from $\Hom_{X_\prin}(\pi_1(\MS,Q_{\MS}),\GG)_\prin$
onto
$\Hom(\pi_1(\MS,Q'_{\MS}) ,\GG)_{\prin}$.
\end{prop}

The three propositions imply the independence statement in \cref{singleout}
above as well as that in \cref{count} above.

\section{Stratified Symplectic Structure}
\label{strati}
Recall that, by definition, the algebra $C^{\infty}(\Hom(\GR,\GG)_\prin)$
of \emph{Whitney smooth functions} on $\Hom(\GR,\GG)_\prin$
(relative to the embedding into $\Hom(F,\GG)$)
is  the algebra of continuous functions
on $\Hom(\GR,\GG)_\prin$
that are restrictions of smooth functions
on the ambient space $\Hom(F,\GG) \cong \GG^{2\ell}$;
cf.~\cite{MR1501735},~\cite{MR0095844}. 
Let 
$C^{\infty}(\mathrm{Rep}(\GR,\GG)_\prin)=
C^{\infty}(\Hom(\GR,\GG)_\prin)^\GG$
 be the algebra of $\GG$-invariant functions in 
$C^{\infty}(\Hom(\GR,\GG)_\prin)$,
viewed as an algebra of continuous functions on the orbit space
$\mathrm{Rep}(\GR,\GG)_\prin$
in an obvious manner.
The algebras 
$C^{\infty}(\Hom(\GR,\GG)_\prin)$
and
$C^{\infty}(\mathrm{Rep}(\GR,\GG)_\prin)$
are what has come to be known as a {\em smooth structure\/}
on
$\Hom(\GR,\GG)_\prin$ and
$\mathrm{Rep}(\GR,\GG)_\prin$,
respectively.

Recall that a {\em differential space\/}
is a topological space $T$ together with a subalgebra $C$ of the algebra
of continuous functions on $T$ subject to certain conditions \cite{MR0467544}.
A {\em smooth structure\/}
on a topological space $T$ is an algebra of functions
that turns $T$ into a differential space,
subject to suitable additional conditions
depending on context.

As \( \mathcal B_{\prin, Q} \) is a Fr\'echet manifold,
it acquires a
natural algebra $C^{\infty}(\mathcal B_{\prin,Q})$ of smooth 
functions on $\mathcal B_{\prin,Q}$.
Furthermore, since the Wilson loop mapping 
\( \wilson\colon \mathcal B_{\prin, Q} \to \Hom(F, \GG) \), 
cf. \eqref{wilson2},
is smooth,
the image 
$\wilson^*(C^{\infty}(\Hom(F,G))) \subseteq C^{\infty}(\mathcal B_{\prin,Q})$
yields a smooth structure on $\mathcal B_{\prin,Q}$ as well,
but not every smooth function on \( B_{\prin, Q} \) arise as
the pull back, under the Wilson loop mapping,  of a smooth function 
on $\Hom(F,G)$.
When we consider the space $\mathcal {CYM}_{\prin,Q}$, the situation changes:

\begin{thm}
The  smooth structure on $\mathcal {CYM}_{\prin,Q}$ that arises via
restriction of functions in $C^{\infty}(\mathcal B_{\prin,Q})$
coincides with the smooth structure that arises via
restriction of functions in $\wilson^*(C^{\infty}(\Hom(F,\GG))) $.
In other words, the Wilson loop mapping \( \wilson^{\sharp}\colon 
\mathcal {CYM}_{\prin,Q} 
\to
\Hom_{X_\prin}(\Gamma,\GG)_\prin \),
cf.~\eqref{singleout2}, is an isomorphism of smooth  spaces.
\end{thm}

\begin{proof} This is a consequence of the Fr\'echet slice analysis
in the proof of \cref{lem2}. 
\end{proof}

Accordingly, we take the  smooth structure
$C^{\infty}(\mathcal {CYM}_{\prin,Q})$
on $\mathcal {CYM}_{\prin,Q}$
to be the algebra of continuous functions on $\mathcal {CYM}_{\prin,Q}$
that arise from ordinary smooth functions
on $\Hom(F,\GG)$ via pull back under the Wilson loop mapping.
Moreover, we take the smooth structure
$C^{\infty}(\mathcal {CYM}_{\prin})$ on
$\mathcal {CYM}_{\prin}$ to be the algebra
\[
C^{\infty}(\mathcal {CYM}_{\prin})=
 C^{\infty}(\mathcal {CYM}_{\prin,Q})^\GG
\]
of $\GG$-invariant functions,
viewed as an algebra of continuous functions on 
$\mathcal {CYM}_{\prin}=\GG \backslash\mathcal {CYM}_{\prin,Q}$
in the obvious way. 
By construction, the homeomorphism \( \wilson_\sharp \), 
cf.~\eqref{singleout2},
induces isomorphisms
\begin{align}
C^{\infty}(\mathcal {CYM}_{\prin,Q})
&\longrightarrow
C^{\infty}(\Hom(\GR,\GG)_\prin)
\\
C^{\infty}(\mathcal {CYM}_{\prin})
&\longrightarrow
C^{\infty}(\mathrm{Rep}(\GR,\GG)_\prin)
\end{align}
of real algebras.

The {\em extended moduli space\/} construction developed in
\cite{MR1370113}, 
\cite{MR1470732},
\cite{MR1277051},
see
\cite{MR1938554}
for a leisurely introduction and more references,
yields the stratified symplectic structure 
\[
(C^{\infty}(\mathrm{Rep}(\GR,\GG)_\prin),\pois)
\]
on
$\mathrm{Rep}(\GR,\GG)_\prin$ determined by the data.
On each stratum of
$\mathrm{Rep}_{\prin}(\Gamma,\GG)$,
that stratified symplectic structure comes down to 
a symplectic structure. 
On the other hand,
the symplectic structure that was
constructed
in~\cite{atibottw} from the familiar expression
\[
\omega(\alpha,\beta)=\int_{\Sigma}\alpha \wedge \beta,
\ \alpha,\beta \in \form^1(\Sigma,\adx)=\mathrm T_{\conn}(\Conn),
\ A \in \Conn,
\]
by infinite dimensional symplectic reduction
involving Sobolev space techniques
induces a 
symplectic structure on the top stratum of
\( \mathcal {CYM}_{\prin} \);
here 
the exterior product $\alpha \wedge \beta$ of the two
$\adx$-valued $1$-forms
$\alpha$ and $\beta$ on $\Sigma$ is evaluated in terms of the
inner product on $\gg$.
With respect to these symplectic structures,
the Wilson loop mapping induces a symplectomorphism from 
the top stratum of $\mathrm{Rep}_{\prin}(\Gamma,\GG)$
onto the top stratum of \( \mathcal {CYM}_{\prin} \).
Suitably interpreted,
the Wilson loop mapping also induces a symplectomorphism from 
each stratum of $\mathrm{Rep}_{\prin}(\Gamma,\GG)$
onto the corresponding stratum of \( \mathcal {CYM}_{\prin} \).

\begin{rema}\label{further}
We can even push a bit further: 
In \cite{MR1638045},
the authors reworked the extended moduli space construction 
and introduced the terminology {\em quasi-hamiltonian $\GG$-space\/}.
Via the Wilson loop mapping \eqref{wilson2},
the \( 2 \)-form on $\Hom(F,\GG)$ constructed in
\cite{MR1277051},
\cite{MR1370113}, 
\cite{MR1470732},
pulls back to a \( 2 \)-form 
on $\mathcal B_{\prin,Q}$ but,
beware, this \( 2 \)-form is not closed.
We can then view the composite
\[
r \circ 
\wilsonn \colon \mathcal B_{\prin,Q} \longrightarrow \GG
\]
as a Lie group valued momentum mapping with respect to the
pulled back 2-form on  $\mathcal B_{\prin,Q}$.
Then $\mathcal B_{\prin,Q}$
together with the $\GG$-action on $\mathcal B_{\prin,Q}$
and the map $r \circ \wilsonn$ is a quasi-hamiltonian $\GG$-space
except that on  $\mathcal B_{\prin,Q}$ 
(as opposed to $\Hom(F,\GG)$) 
the non-degeneracy condition is to be interpreted appropriately.
Reduction with respect to the structure group $\GG$ then yields 
the moduli space of central Yang-Mills connections
as a stratified symplectic space.
Suitably interpreted, we can view  $\mathcal B_{\prin,Q}$
as the reduced space relative to the group of based gauge transformations.
This is, perhaps, not strictly correct -- we did not make
precise the momentum mapping for the  group of based gauge 
transformations -- and  
the purported space that arises by reduction 
relative to the group of based gauge transformations
would have to be a proper subspace
of  $\mathcal B_{\prin,Q}$.
However, to use some physics terminology,
we are interested only in what happens {\em on shell\/}, that is,
on the subspace $\mathcal {CYM}_{\prin,Q}$,
and we can interpret \cref{singleout} by saying that,
on shell, the Fr\'echet manifold $\mathcal B_{\prin,Q}$
has the appropriate behavior as if it did arise by reduction 
relative to the group of based gauge transformations.
\end{rema}

\section{Fr\'echet Slices}
\label{frechetslices}
Our references for terminology
and notation in the framework of infinite-dimensional
differential geometry
are \cite{MR656198} for the Fr\'echet case
and  \cite{Neeb2006}
for the
locally convex setting.
Thus we  freely use the terms
{\em smooth manifold\/},
{\em Lie group\/},
{\em Fr\'echet structure\/} etc.
without further explanation.

\subsection{Principal Lie subgroups}
In finite dimensions, for any Lie subgroup \( H \subseteq G \), 
the space of left cosets \( G \slash H \) carries 
the structure of a
smooth manifold that turns the natural projection 
\( \pi\colon G \to G \slash H  \) into a right principal \( H \)-bundle. 
Recall that the local diffeomorphism 
\begin{equation}
  \mu\colon \liea{g} 
=  \liea{k} \oplus \liea{h} \to G, \qquad (X,Y) \mapsto \exp(X) \exp(Y),
\ X \in \liea{k},\ Y \in \liea{h},
\end{equation}
plays an important role in the construction of charts on \( G \slash H \). This 
concept needs some refinement for general 
locally convex manifolds since in general
\( \mu \) fails to be a local diffeomorphism 
(and the exponential map  need not 
exist in the first place).

\begin{prop}[{\cite[Proposition 7.1.21]{Gloecknernneeb2013}}] \label{prop::subliegroup:principalLieSubgroup}
  Let \( G \) be a 
Lie group and \( H \subseteq G \) a subgroup. Given a Lie group structure on \( H \), the statements {\rm(1)} and {\rm(2)} below  are equivalent:
 
{\rm(1)}     The group \( H \) is a splitting Lie subgroup, and there exists a unique 
smooth structure on the space of left cosets \( G \slash H \) such that the canonical 
projection \( \pi\colon G \to G \slash H \) defines a smooth principal \( H \)-bundle 
structure. In this case, a map \( f\colon G \slash H  \to \GGGG \) to some 
manifold \( \GGGG \) 
is smooth if and only if the induced map 
\( \hat f = f \circ \pi\colon G \to \GGGG \) is smooth.
    
{\rm(2)} The inclusion \( \iota\colon H \to G \) is a morphism of Lie groups, and 
there exist an open subset \( V \) containing \( 0 \) in some locally convex 
space 
\( \liea{k} \) and a smooth map \( \sigma\colon V \to G \) 
with \( \sigma(0)=e \) such that
      \begin{equation}
        \mu\colon V \times H \to G, \qquad (X, y) \mapsto \sigma(X) y,
\ 
X \in V,\ y \in H,
      \end{equation} 
      is a diffeomorphism onto an open subset of \( G \). In this case, 
\( \mu( V \times H) \) is a tube around \( H \) in \( G \). 

\end{prop}

Direct inspection of the proof of that
proposition in \cite{Gloecknernneeb2013} shows 
the following.

\begin{comp}
The equivalence of {\rm(1)} and {\rm(2)} in \cref{prop::subliegroup:principalLieSubgroup}
is 
in particular valid
in the category of  tame smooth Fr\'echet manifolds
and tame smooth maps.
\end{comp}

Given a Lie group $G$ and a subgroup $H$,
we refer to \( H \) as a 
{\em principal Lie subgroup\/}
when  one (and hence both) of the conditions 
in \cref{prop::subliegroup:principalLieSubgroup} are met.
In~\cite{Gloecknernneeb2013} such a subgroup $H$ is termed a
{\em split Lie subgroup\/} 
but, as the requirement of being a splitting 
submanifold is not enough to ensure the properties
in the above proposition 
we prefer the terminology 
principal Lie subgroup.
When the Lie group $G$ admits an exponential map, 
a local product structure around 
the identity already suffices since the local product structure 
can then be extended to the whole subgroup.

The subsequent proposition is a consequence of
\cite[Proposition 7.1.24]{Gloecknernneeb2013}.
Since the final version of this book is not yet available, 
for ease of exposition,
we spell out a complete statement.

\begin{prop} 
\label{prop::lcliegroup:principalLieSubgroupByExpLocalDiffeomorphism}
  Let \( G \) be a Lie group with smooth exponential map and 
\( H \subseteq G \) a splitting Lie subgroup. Denote a complement of 
\( \liea{h} \) in \( \liea{g} \) by \( \liea{k} \), and 
let \( V \subseteq \liea{k} \) be an open neighborhood of $0$ in 
\( \liea{k} \). If the 
map 
  \begin{equation}
    \mu\colon V \times H \to G, \qquad (X, y) 
\mapsto \exp(X) y,\ X\in V,\ y \in H,
  \end{equation} 
  is a local diffeomorphism at \( (0, e) \), then \( H \) is a 
principal Lie subgroup.
\end{prop}
\begin{proof}
  Since \( \mu \) is a local diffeomorphism at \( (0, e) \) there exists a 
neighborhood \( U_H \subseteq H \) 
of $e$ in $H$ 
such that the restriction \( \mu_{\restriction V \times U_H} \) is a 
diffeomorphism onto an open neighborhood \( U_G \subseteq G \) 
of $e$ in $G$
(after shrinking \( V \) if need be). Because of the identity 
\[
\mu(X, q y) = \mu(X, q) y,\ X\in V,\ q,y \in H,
\]
the map \( \mu \) is a local diffeomorphism at every point 
of  $V \times H$. 

Thus, to apply \cref{prop::subliegroup:principalLieSubgroup}, 
it suffices to show that \( \mu \) is injective. 
Since \( U_H \) is open in \( H \), there exists an open 
neighborhood $W_G$ of $e$  in $G$  such that \( U_H = H \cap W_G \). 
By  shrinking \( V \) further, we can 
arrange for the set \( \exp(-V) \exp(V) \) to lie completely in \( W_G \). 
Now let \( (X,a) \) and \( (Y,b) \) be two points of $V\times H$ 
having the same image 
under \( \mu \). Then 
  \begin{equation}
    \exp(-Y) \exp(X) = b a^{-1}.
  \end{equation}
  The left-hand side lies in \( W_G \) by assumption and the 
right-hand side is an element of \( H \)
whence both expressions are contained in \( U_H \). On the other hand, 
\( \mu \) is bijective on \( V \times U_H \), 
and hence the calculation 
\[ 
\mu(Y, b a^{-1}) = \mu(Y,b) a^{-1} = \mu(X,a) a^{-1} = \mu(X,e)  
\]
establishes the desired result \( X=Y \) and \( a=b \). 
\end{proof}

Since the derivative 
\( \mu'_{0,e}\colon \liea{k} \times \liea{h} \to \liea{g} \) 
of $\mu$ at the point $(0,e)$ of $V \times H$
is just the direct sum isomorphism 
\( \liea{k} \oplus \liea{h} \to \liea{g} \), the inverse function theorem 
yields the necessary local diffeomorphism 
so that \cref{prop::lcliegroup:principalLieSubgroupByExpLocalDiffeomorphism} 
applies
to Banach Lie groups; 
thus all splitting Lie subgroups of a Banach Lie group are principal.

Exploiting the Nash--Moser theorem, 
we obtain the following important 
consequence 
for tame Fr\'echet Lie groups. 

\begin{thm} 
\label{prop::lieGroup:finiteDimSubgroupIsStrongSplittingLieSubgroup} 
  Let \( G \) be a tame Fr\'echet Lie group with tame smooth exponential map 
which is a local diffeomorphism close to  \( 0 \). Every splitting Lie 
subgroup \( H \subseteq G \) of finite dimension or finite codimension is a 
principal Lie subgroup. 
\end{thm}
\begin{proof}
  Denote the complement of 
\( \liea{h} \) in \( \liea{g} \) by \( \liea{k} \), 
and let \( V \subseteq \liea{k} \) be an open neighborhood
of $0$. In view of \cref{prop::lcliegroup:principalLieSubgroupByExpLocalDiffeomorphism} 
above, we have to show that
  \begin{equation}
    \mu\colon V \times H \to G, \qquad (X, y) \mapsto \exp(X) y,
\ X \in V, y \in H,
  \end{equation} 
  is a local diffeomorphism at \( (0, e) \).

The derivative of \( \mu \) at the point 
\( (X,e) \) of $V \times H$ is given by:
  \begin{equation}
    \mu'_{(X,e)}\colon \liea{k} \times \liea{h} \to \TBundle_{\exp(X)} G, 
\quad (Y, Z) \mapsto \exp'_X Y + (\LeftTrans_{\exp(X)})'_e Z,
\ Y \in \liea{k},\ Z\in \liea{h}.
  \end{equation}
  In the case of a finite-dimensional subgroup \( H \), apply 
\cref{prop::tameFrechet:extendFiniteDimSplitting} below
to 
  \begin{equation}
    (\exp^{-1} \circ \mu)'_{(X,e)}\colon \liea{k} \times \liea{h} 
\to \liea{g}, \quad (Y, Z) \mapsto 
Y + (\exp^{-1} \circ \LeftTrans_{\exp(X)})'_e Z,
\ Y \in \liea{k},\ Z\in \liea{h},
  \end{equation}
and conclude that 
\( \mu'_{(Y,e)} \) is invertible with tame inverse for every 
\( Y \in V \) near \( X \). By \( H \)-translation the same statement holds 
true at every point \( (Y, h) \) near \( (X, e) \). 
By the version of the
Nash--Moser inverse function 
given in~\cite[Theorem~3.4.2]{diez2013}, 
we conclude that \( \mu \) is a local diffeomorphism at \( (0, e) \).

A similar reasoning, 
applied to the map
  \begin{equation}
    (\LeftTrans_{\exp(-X)} \circ \mu)'_{(X,e)}\colon 
\liea{k} \times \liea{h} \to \liea{g}, 
\qquad (Y, Z) \mapsto (\LeftTrans_{\exp(-X)} \circ \exp)'_X Y + Z,
\ Y \in \liea{k},\ Z\in \liea{h}, 
  \end{equation}
establishes the case of finite codimension.
\end{proof}

For intelligibility, we now recall the tame Fr\'echet setting and
the Nash--Moser inverse function theorem,
cf.~\cite{MR656198, diez2013}.
A Fr\'echet space is called graded if it carries a distinguished increasing fundamental system of seminorms \( \normDot_k \).
If the growth of the sequence \( k \mapsto \normDot_k \) is controlled by the exponential map, then the graded Fr\'echet space is called \emph{tame} (see \cite[Definition~II.1.3.2]{MR656198} for the exact statement).
A smooth map \( f \) between graded Fr\'echet spaces is \emph{\( r \)-tame smooth} if \( f \) and all its derivatives satisfy a local estimate of the form
\begin{equation}
	\norm{f(x)}_k \lesssim 1 + \norm{x}_{k+r}.
\end{equation}
Roughly speaking, this means that \( f \) has a \lq maximal loss of \( r \) derivatives\rq.

\begin{thm}[{\cite[section III.1]{MR656198}}] \label{prop::locallyConvexSpace:NashMoserInverseTheorem} 
	Let $X$ and $Y$ be tame Fr\'echet 
spaces,
let $U$ be an open subset of $X$, and
let $f: U \to Y$ be tame smooth. 
	Suppose that the derivative $f'$ has a tame smooth family 
$\Psi^f$ of inverses, that is, $\Psi^f: U \times Y \to X$ is a 
tame smooth map, and the family $\Psi^f_x: Y \to X$ is inverse to 
$f'_x$ for every $x \in U$. Then the map $f$ is 
locally bijective and the inverse is a tame smooth map.
\end{thm}

\begin{prop}[{\cite[pp.~47ff]{MR2634197}}]  
\label{prop::tameFrechet:extendFiniteDimSplitting}
  Let \( A \) and \( X \) be tame Fr\'echet 
spaces, and let
\( E \subseteq X \) and \( F \subseteq X \) be closed subspaces. 
Suppose that \( E \) is finite-dimensional. Moreover, let 
\( \Phi\colon A \times (E \times F) \to X \) 
be a tame smooth family of linear maps which decomposes into
  \begin{equation}
    \Phi_a (u, v) = \varphi_a(u) + v, \quad a\in A,\ u\in E,\ v \in F,
  \end{equation}
  where \( \varphi\colon A \times E \to X \) is a tame smooth family of 
injective, linear maps. If,
for some \( a_0 \in A \), the partial map
 \( \Phi_{a_0}\colon E \times F \to X \) 
is a linear and topological isomorphism 
\(  E \oplus F \to X\)
then there is an open neighborhood 
\( U \subseteq A \) of \( a_0 \) in $\conn$
such that \( \Phi_a \) is a bijection for every \( a \in U \) 
in such a way that the inverses constitute 
a tame smooth map \( U \times X \to E \times F \).  
\end{prop}

\subsection{Definition of a slice}
In finite  dimensions, a smooth retraction and a tubular 
neighborhood are equivalent to the existence of a slice, 
see~\cite[Theorem~2.3.26]{MR2021152}.
However, the proof relies on the inverse 
function theorem; indeed, this theorem entails that
an infinitesimal splitting generates a local product decomposition. 
This argument 
does not carry over to arbitrary 
infinite-dimensional manifolds. 
For this reason and for later reference, we  now make precise 
the notion of slice we subsequently use; 
 this is the strongest concept 
of a slice
we
 could possibly think of. 

\begin{defn}
  Let \( \Upsilon\colon G \times \GGGG \to \GGGG \) be a smooth action of a 
Lie group \( G \) on a manifold \( \GGGG \). 
Let $q$ be a point of $\GGGG$ and
suppose that the stabilizer 
\( G_q \subseteq G \) at \( q \in \GGGG \) is a principal Lie subgroup. 
A \emph{slice} at \( q \in \GGGG \) is a submanifold \( S \) 
of \( \GGGG \) that contains \( q \) in such a way that the following
conditions are met:
  \begin{enumerate}[label=(S\arabic*),ref=(S\arabic*)]
    \item \label{i::slice:SliceDefSliceSweepIsOpen} 
      The $G$-closure 
\( G \cdot S \) of $S$ is an open neighborhood of the orbit \( G \cdot q \) 
in $\GGGG$ and \( S \) is closed in \( G \cdot S \). 

    \item \label{i::slice:SliceDefSliceInvariantUnderStab}
      The submanifold 
\( S \) is closed under the induced action of \( G_q \) in the sense that
$G_q \cdot S \subseteq S$. 

    \item \label{i::slice:SliceDefOnlyStabNotMoveSlice}
      Any $x \in G$ such that 
\( (x\cdot S) \cap S \neq \emptyset \) necessarily lies in \( G_q \). 

    \item \label{i::slice:SliceDefLocallyProduct}
     The smooth submersion $G \to G/G_q$ admits
      a local section \( \chi\colon  U \to G \) defined on an open 
neighborhood \( U \subseteq G/G_q\) of the identity coset in such a way 
     that the map
      \begin{equation*}
        \chi^S\colon U \times S \to \GGGG, \qquad (\equivClass{x}, y) \mapsto \chi(\equivClass{x}) \cdot y,\ x\in U,\ y\in S,
      \end{equation*}
      is a diffeomorphism onto an open neighborhood \( V  \) of  
\( q \) in $\GGGG$. 
  \end{enumerate}
A Lie group action \( \Upsilon\colon G \times \GGGG \to \GGGG \) 
is said to \emph{admit a slice at} \( q \in \GGGG \) if the stabilizer \( G_q \) 
of $q$ in $G$ 
is a principal Lie subgroup of \( G \) and there is a slice at \( q \in \GGGG \).
\end{defn}

\begin{prop}
\label{ps}
  Let \( \Upsilon\colon G \times \GGGG \to \GGGG \) be a smooth Lie group action 
admitting a slice \( S \) at the point \( q \) of $\GGGG$. The following hold: 
  \begin{enumerate}
    \item[{\rm(1)}] \label{prop::slice:mHasMaximalStabilizerOfWholeSlice} 
      Given \( y \in S \), the stabilizer \( G_y\) of $y$ is a subgroup
of the stabilizer \( G_q \) of $q$. That is, the point 
\( q \) has the maximal 
stabilizer of the whole slice.
    
    \item [{\rm(2)}]
      The point $q$ has a neighborhood \( V \) in \( \GGGG \) such that any 
   point of \( V \) has a stabilizer that is conjugate to a subgroup of 
$G_q$. \qedhere
  \end{enumerate}
\end{prop}
\begin{proof}
  The first claim follows directly from property \iref{i::slice:SliceDefOnlyStabNotMoveSlice}. 

  Let \( V \) be a neighborhood of the kind characterized in 
\iref{i::slice:SliceDefLocallyProduct}. Since every \( v \in V \) is 
obtained by translation of the slice \( S \) along the orbit, 
for some \( x \in G \), the point \( x\cdot v\) lies in $S$. Applying the first 
statement
 together with the equivariance of stabilizer subgroups 
we  conclude  that
 \( x G_v x^{-1} \subseteq G_q \) for some $x\in G$.
\end{proof}

\subsection{Slice for the action of a subgroup}
Let \( (\GGGG, G) \) be a \( G \)-manifold,
and let $q$ be a point of $\GGGG$ 
having the property that the action 
admits a slice at \( q  \). In this section we discuss how one can
 construct a slice at \( q \) for the induced action of a 
subgroup \( H \) of \( G \).

\begin{prop} \label{slice::actionOfSubgroup}
  Let \( \GGGG \) be a manifold with a smooth \( G \)-action \( \Upsilon \) 
admitting a slice at \( q \in \GGGG \). Let \( H \subseteq G \) be a normal 
Lie subgroup. If the induced action of \( H \) on \( \GGGG \) is free and the 
product \( G_q H \) is a principal Lie subgroup of \( G \), 
then there exists  a slice at \( q \) for the \( H \)-action as well.
\end{prop}

In the gauge theory context in \cref{frechet} above, 
we apply this proposition with
$G$ being the group  $\Gau$  of all gauge transformations and 
\( H \) the subgroup  \( \GauQ \) of based gauge transformations
with respect to the chosen base point $Q$ of the manifold written
in \cref{frechet} as
$\GGGG$ (but presently the notation $\GGGG$ refers to a $\GG$-manifold). 
It is a standard fact that \( \GauQ \) acts freely on the space of 
connections and that it is a normal
subgroup of \( \Gau \). 
By \cref{prop::lieGroup:finiteDimSubgroupIsStrongSplittingLieSubgroup},
$\GauQ$ is a principal Lie subgroup of $\Gau$.
Furthermore, 
given a smooth connection $\conn$ on $\prin$,
the stabilizer  $\GauA$
of $\conn$ in $\Gau$
is finite-dimensional, and
the group \( \GauA \GauQ \) 
has finite codimension in $\Gau$ (indeed, \( \dim(G) - \dim(\GauA) \)) 
and hence, by 
\cref{prop::lieGroup:finiteDimSubgroupIsStrongSplittingLieSubgroup},
the group \( \GauA \GauQ \) 
is a principal Lie subgroup of \( \Gau \).

\begin{proof}
  Let \( S \) be a slice at \( q \in \GGGG \) for the \( G \)-action. 
The group
\( G_q H \) is a principal Lie subgroup of \( G \) by assumption. Hence there 
is a closed subspace \( \liea{t}_q \) of \( \liea{g} \), an open neighborhood 
$V$ of $0$ in \(\liea{t}_q \),  and a smooth map 
\( \sigma\colon V \to G \) with \( \sigma(0) = e \) such that the map
  \begin{equation*}
    \mu\colon V \times G_q H \to G, \qquad (X, xy) \mapsto \sigma(X) \, xy,
\ x\in G_q,\ y \in H, 
  \end{equation*} 
  is a diffeomorphism onto an open subset of \( G \). 
Since $H$ is a normal subgroup of $G$,
\[
 H \sigma(V) = \sigma(V) H = \mu(V, H).
\]
  We claim that \( \hat{S} = \sigma(V) \cdot S \) is a slice for the 
\( H \)-action. Indeed, \( \hat{S} \) is a submanifold of 
\( \GGGG \) containing \( q \) since,
in view of 
\cref{ps} (1),
for any point 
\( y \) of \( S \), 
the stabilizer \( G_y \) of $y$ is a subgroup of \( G_q\).
Consequently the restriction of the action 
\( \Upsilon \) to \( \sigma(V) \times S \) is a diffeomorphism onto 
\( \hat{S} \); 
we note that \( \sigma(X) \) lies in \( G_q \) 
only for the trivial element \( X = 0\). Furthermore, all 
defining properties of a slice are 
met:

\noindent
(\^S1) The $H$-closure 
   \( H \cdot \hat{S} = H \sigma(V) \cdot S = \mu(V, H) \cdot S \) 
     is open in \( \GGGG \) since \( G \cdot S \subseteq \GGGG \) is open and 
     since \( V \) is an open subset of \( \liea{t}_q \). 
      Furthermore, since
       \( V \times \set{e} \) is closed in \( V \times H \), 
      the space \( \hat{S} = \mu(V, e) \cdot S \) 
       is closed in \( H \cdot \hat{S} = \mu(V, H) \cdot S \).

\noindent
(\^S2) The manifold 
      \( \hat{S} \) is clearly invariant under the trivial stabilizer \( H_q = \set{e} \).

\noindent
(\^S3) Let  \( x \in H \) and \( y \in \hat{S} \) such that \( x \cdot y \) 
lies in \( \hat{S} \) again. We have to show \( x = e \). By the construction 
  of \( \hat{S} \) there are \( X, X' \in V \) 
   and \( s, s' \in S \) with \( x \sigma(X) \cdot s = \sigma(X') \cdot s' \).
 Now property \iref{i::slice:SliceDefOnlyStabNotMoveSlice} for the 
\( G \)-slice \( S \) implies that, for some \( w \in G_q \),  
      \begin{equation*}
        x \sigma(X) = \sigma(X) \sigma(X)^{-1} h \sigma(X) = \sigma(X') w.
      \end{equation*}
      Since \( H \) is a normal subgroup of \( G \), the element 
\( \tilde{x} \defeq \sigma(X)^{-1} x \sigma(X) \) lies in 
\( H \) again and hence 
      \begin{equation*}
        \mu(X, e) = \sigma(X) = \sigma(X') w \tilde{y}^{-1} 
= \mu(X', w \tilde{y}^{-1}).
      \end{equation*}
      The map \( \mu\colon V \times G_q H \to G \) is injective and 
thus yields \( X = X' \) and \( w = e\), \(\tilde{x} = e \). Consequently \( x = e \).

\noindent
(\^S4) Since \( S \) is a slice, there is a smooth map 
    \( \chi\colon  U \to G \) 
defined on an open neighborhood \( U \) of \( 0 \) 
in $\liea{t} \times \liea{h}$
such that the maps
      \begin{equation*}\begin{split}
        &\chi^S\colon U \times S \to \GGGG, \qquad (X, y) 
          \mapsto \chi(X) \cdot y,\ X\in U,\ y \in S, \\
        &\nu\colon U \times G_q \to G, \qquad (X, y) \mapsto \chi(X) y,
\  X\in U,\ y \in G_q,
      \end{split}\end{equation*}
      are diffeomorphisms onto open neighborhoods \( W_\GGGG \) 
of \( q \) in $\GGGG$ and \( W_G\) of \( e \) in $G$, respectively. 
Shrink \( V \) (and hence the slice \( \hat{S} \)) and choose an open 
and conjugation invariant  subset \( V_H \) of \( H \) such that 
\( \mu(V, V_H) \subseteq \nu(U, e) \). Then the map
\( \lambda\colon V \times V_H \mapsto U \) defined by 
\( {\lambda = \nu^{-1} \circ \mu} \) 
is a diffeomorphism onto an open subset of \( U \).

      Finally, the map 
      \begin{equation*}
        \eta\colon V_H \times \hat{S} \to \GGGG, 
\qquad (v, \sigma(X) \cdot y) \mapsto v \sigma(X) \cdot y,
\ 
v \in  V_H,\ 
X \in V,\ 
y \in S,
      \end{equation*}
      is a diffeomorphism onto an open neighborhood of \( q \) in \( \GGGG \).
 Indeed, this map can be written as the composite 
      \begin{equation*}\begin{tikzcd}
          V_H \times \hat{S} 
            \arrow{r}{\id \times \mu^{-1}} 
        & V_H \times V \times S
            \arrow{r}{\lambda \times \id}
        & U \times S
            \arrow{r}{\chi^S}
        & W_\GGGG
      \end{tikzcd}\end{equation*}
made explicit by the association
      \begin{equation*}\begin{tikzcd} 
         (v, \sigma(X) \cdot y) 
          \arrow[mapsto]{r}{}
        & (v, X, y) \arrow[mapsto]{r}{}
        & (\nu^{-1} \circ \mu(v, X), y) \arrow[mapsto]{r}{}
        & \mu(v, X) \cdot y,
      \end{tikzcd}\end{equation*}
where 
$v \in  V_H$, $X \in V$, $y \in S$.
\end{proof}

\section*{Acknowledgements}

We are indebted to B. Tumpach and A. Waldron for discussion.
We gratefully acknowledge support by the CNRS, by the
Labex CEMPI (ANR-11-LABX-0007-01), 
by the German National Academic Foundation
(DAAD), and by the Max Planck Institute for Mathematics in the Sciences
(Leipzig).

\bibliographystyle{alpha}

\def\cprime{$'$} \def\cprime{$'$} \def\cprime{$'$} \def\cprime{$'$}
  \def\cprime{$'$} \def\cprime{$'$} \def\cprime{$'$} \def\cprime{$'$}
  \def\dbar{\leavevmode\hbox to 0pt{\hskip.2ex \accent"16\hss}d}
  \def\cprime{$'$} \def\cprime{$'$} \def\cprime{$'$} \def\cprime{$'$}
  \def\cprime{$'$} \def\Dbar{\leavevmode\lower.6ex\hbox to 0pt{\hskip-.23ex
  \accent"16\hss}D} \def\cftil#1{\ifmmode\setbox7\hbox{$\accent"5E#1$}\else
  \setbox7\hbox{\accent"5E#1}\penalty 10000\relax\fi\raise 1\ht7
  \hbox{\lower1.15ex\hbox to 1\wd7{\hss\accent"7E\hss}}\penalty 10000
  \hskip-1\wd7\penalty 10000\box7}
  \def\cfudot#1{\ifmmode\setbox7\hbox{$\accent"5E#1$}\else
  \setbox7\hbox{\accent"5E#1}\penalty 10000\relax\fi\raise 1\ht7
  \hbox{\raise.1ex\hbox to 1\wd7{\hss.\hss}}\penalty 10000 \hskip-1\wd7\penalty
  10000\box7} \def\polhk#1{\setbox0=\hbox{#1}{\ooalign{\hidewidth
  \lower1.5ex\hbox{`}\hidewidth\crcr\unhbox0}}}
  \def\polhk#1{\setbox0=\hbox{#1}{\ooalign{\hidewidth
  \lower1.5ex\hbox{`}\hidewidth\crcr\unhbox0}}}
  \def\polhk#1{\setbox0=\hbox{#1}{\ooalign{\hidewidth
  \lower1.5ex\hbox{`}\hidewidth\crcr\unhbox0}}}

\end{document}